\documentclass[11pt]{amsart}
\usepackage[left=1in,right=1in,bottom=1in,top=1in]{geometry}
\usepackage{amsmath}
\usepackage{amsfonts}
\usepackage{amssymb}
\usepackage{amsthm}
\usepackage{young}
\usepackage[vcentermath]{youngtab}
\usepackage{tikz}
\usetikzlibrary{shapes.misc}
\usepackage{graphicx}
\usepackage{caption}
\usepackage{algorithm}
\usepackage[noend]{algpseudocode}
\usepackage{amsmath, amssymb, amsthm}
\usepackage{blindtext}
\usepackage[pdftex,bookmarks=true]{hyperref}
\usepackage{cleveref}
\usepackage{comment}
\usepackage{enumerate}
\usepackage{fullpage}
\usepackage{mathtools}
\usepackage{subcaption}
\usepackage{tikz}
\usetikzlibrary{decorations.markings}
\usetikzlibrary{shapes.misc}
\usepackage{graphicx}
\usepackage{caption}

\usepackage{enumitem}
\setlist[enumerate]{itemsep=.3mm}
\setlist[itemize]{itemsep=.3mm}

\tikzset{cross/.style={cross out, draw=black, minimum size=2.5*(#1-\pgflinewidth), inner sep=2pt, outer sep=0.5pt},
	cross/.default={1pt}}
\usepackage{xcolor}
\usetikzlibrary{calc,arrows}
\newcommand{\boundellipse}[3]
{(#1) ellipse (#2 and #3)
}

\newcommand\R{\mathbb {R}}
\newcommand\E{\mathbb {E}}
\renewcommand\P{\mathbb {P}}

\newcommand{\ep}{\varepsilon}

\newcommand{\la}{\lambda}
\newcommand{\La}{\Lambda}
\newcommand{\Z}{\mathbb Z}
 \newcommand{\N}{\mathbb N}
 
  \newcommand{\cp}{\text{CP}}
\newcommand{\Pcp}{\mathbb{P}_{\textsc{cp}}}
\newcommand{\Pgw}{\mathbb{P}_{\textsc{gw}}}

\renewcommand{\deg}{\text{deg}}
\newcommand{\one}{\textnormal{\textbf{1}}}
\newcommand{\zero}{\textnormal{\textbf{0}}}

\newcommand{\gw}{\textsf{\footnotesize{\textbf{GW}}}}
\newcommand{\GW}{\textsf{\footnotesize{\textbf{GW}}}}
\newcommand{\gwc}{\textsf{\footnotesize{\textbf{GWC}}}}
\newcommand{\egw}{\textsf{\footnotesize{\textbf{EGW}}}}

\newcommand{\CP}{\textsc{cp}}

\newcommand{\Exp}{\textnormal{Exp}}

\newcommand{\bm}{\mathbf{M}}
\newcommand{\bs}{\mathbf{S}}
\newcommand{\bu}{\mathbf{U}}

\newcommand{\cH}{\mathcal H}
\newcommand{\wt}{w}

\theoremstyle{plain}
\newtheorem{theorem}{Theorem}[section]

\newtheorem{lemma}[theorem]{Lemma}
\newtheorem{proposition}[theorem]{Proposition}
\newtheorem{definition}[theorem]{Definition}
\newtheorem{remark}[theorem]{Remark}

\renewcommand{\ep}{\varepsilon}
\newcommand{\T}{\mathcal T}

\newcommand{\whp}{\textsf{whp}}
\newcommand{\A}{\mathcal A}
\newcommand{\Pois}{\text{Pois}}
\newcommand{\prf}{\textit{Proof.}\hspace{2mm}}
\newcommand{\no}{\nonumber\\}
\newcommand{\bS}{\textbf{S}}

\renewcommand{\L}{\mathcal L}

\makeatletter
\def\l@subsection{\@tocline{2}{0pt}{2.5pc}{5pc}{}}
\makeatother

\begin{document}
	\title{Subcritical epidemics on random graphs}
	\author{Oanh Nguyen}
	\author{Allan Sly}
	\address{Division of Applied Mathematics\\ Brown University\\  Providence, RI 02906}
	\email{oanh\_nguyen1@brown.edu}
	\address{Department of Mathematics\\ Princeton University\\  Princeton, NJ 08544}
	\email{asly@princeton.edu}
	\thanks{Nguyen is supported by NSF grants DMS-1954174 and DMS-2246575. Sly is partially supported in part by NSF grant DMS-1855527, a Simons Investigator grant and a MacArthur Fellowship.}
	
	\maketitle
\begin{abstract}
We study the contact process on random graphs with low infection rate $\lambda$. For random $d$-regular graphs, it is known that the survival time is $O(\log n)$ below the critical $\lambda_c$.  By contrast, on the Erd\H{o}s-R\'enyi random graphs $\mathcal G(n,d/n)$, rare high-degree vertices result in much longer survival times. We show that the survival time is governed by high-density local configurations. In particular, we show that there is a long string of high-degree vertices on which the infection lasts for time $n^{\lambda^{2+o(1)}}$.  To establish a matching upper bound, we introduce a modified version of the contact process which ignores infections that do not lead to further infections and allows for a sharper recursive analysis on branching process trees, the local-weak limit of the graph. Our methods, moreover, generalize to random graphs with given degree distributions that have exponential moments.
\end{abstract}
\tableofcontents
	\section{Introduction}
Stochastic processes on random networks can be affected by both global properties of the graph like expansion but also local properties like the local-weak limit and the presence of rare local neighbourhoods.  Our goal in this paper is to investigate what properties of sparse Erd\H{o}s-R\'enyi random graphs determine the survival time for infection processes.

	 The contact process is a model for the spread of infections on a graph $G = (V, E)$. For a given $\lambda>0$, the contact process $(X_t)_{t\ge 0}$ with infection rate $\la$ and recovery rate $1$ is a continuous-time Markov chain, with state space $\{0,1\}^V$, tracking whether a vertex $v$ is either infected ($X_t(v)=1$) or healthy ($X_t(v)=0$). The process evolves according to the following rules:
	\begin{itemize}	
		\item Each infected vertex infects each of its neighbors independently at Poisson rate $\lambda$ and is healed at Poisson rate $1$;
		
		\item Infection and recovery events in the process happen independently.
	\end{itemize}

On finite graphs, the contact process has a single absorbing state which is no infections and we define the time to reach it as the \emph{survival time}.  The primary question  is how quickly this occurs and how it depends on $\lambda$. We can ask if the behaviour on Erd\H{o}s-R\'enyi random graphs is the same as for random $d$-regular graphs. In terms of the threshold of $\la$ between fast (polynomial) and slow (exponential) survival times, the answer is yes, both are at approximately $\frac1{d}$ (see \cite{NNS, p92}) which is determined by the global branching rate of the process.
When $\lambda$ is small, the contact process on random regular graphs survives for time $O(\log n)$ (see \cite{ls17,mv16}). One can ask whether the survival time is still logarithmic for other sparse random graphs, in particular the Erd\H{o}s-R\'enyi random graph $\mathcal G(n, d/n)$?
The answer turns out to be negative for a surprisingly simple reason. Since the largest degree in an Erd\H{o}s-R\'enyi random graph is asymptotically $\frac{\log n}{\log \log n}$, this directly gives a lower bound of $\exp(\frac{c\log n}{\log \log n})$ which far exceeds $\log n$.

It is then natural to ask whether $\exp(\frac{c\log n}{\log \log n})$ gives the correct order of the survival time.  The convergence to equilibrium of the high-temperature Glauber dynamics of the Ising model gives an interesting comparison.  Its mixing time is determined by coupling the all-plus and all-minus initial conditions and the set of disagreements can be thought of as a kind of infection process where updates can remove disagreements or spread them to neighbours. In \cite{AllanExact}, it was shown that the mixing time of the Glauber dynamics for the high-temperature continuous-time ferromagnetic Ising model on $\mathcal G(n, d/n)$ is $\exp(\frac{\Theta(1)\log n}{\log \log n})$.  It is essentially determined by the time to mix around the maximal-degree vertices in the graph.  We show, by contrast, that the contact process survives for time polynomial in $n$, much longer than would be predicted by just the high-degree vertices. In the following theorem we, moreover, determine the exponent of the polynomial up to the log factors.
		\begin{theorem} \label{thm:ER}
	Let $d>1$ and $G_n\sim \mathcal G(n, d/n)$.	There exist positive constants $\lambda_*(d), C(d)$ such that the contact process on $G_n$ starting at all vertices infected and infection rate $\lambda < \lambda_*$ has survival time $T_{\la, n}$ that satisfies
		\begin{equation}
		n^{c\la^{2}/\log^{2}(\la^{-1})} \le T_{\la, n}\le n^{C\la^{2}\log(\la^{-1})} \qquad  \nonumber
		\end{equation}
		 with probability $1-o(1)$ as $n\to \infty$ where $c=10^{-9}$.
	\end{theorem}
This result is a special case of our more general results on random graphs with given degree distributions described later in the section.  As we will describe, the survival time comes from large connected components of vertices of large constant degree.  The discrepancy between the contact process and the Glauber dynamics can be explained by the existence of a phase for the contact process in dimension one but not for the Glauber dynamics.  As a result, it is different rare local structures which determine the survival times.

\subsection{Background}\label{bg} The contact process was first introduced by Harris \cite{harris74} who studied the process on the lattice $\Z^{d}$. Among other things, he studied the phase diagrams of the contact process which since then has attracted intensive research. For an infinite rooted graph $G$, there are three phases that are of particular interest:
 \begin{itemize}
 	\item (\textit{Extinction}) the infection becomes extinct in finite time almost surely;
 	
 	\item (\textit{Weak survival}) the infection survives forever with positive probability, but the root is infected only finitely many times almost surely;
 	
 	\item (\textit{Strong survival}) the infection survives forever and the root gets infected infinitely many times with positive probability.
 \end{itemize}
 In that sense, we denote the extinction-survival threshold by
 $$\lambda_1(G) = \inf\{\lambda: (X_t) \text{ survives forever with positive probability}\}$$
 and the weak-strong survival threshold by
 $$\lambda_2(G) = \inf\{\lambda: (X_t) \text{ survives strongly}\}.$$

 For lattices, when the origin is initially infected, it is well-known that there is no weak survival phase (see \cite{harris74}, Bezuidenhout-Grimmett \cite{bezuidenhout1990critical}, the books of Liggett \cite{liggett:sis}, \cite{liggett:ips} and the references therein).

Beyond lattices, another natural family of graphs to study the contact process is the family of infinite trees. For the infinite $d$-regular tree $\mathbb T_{d}$ with $d\ge 3$, the contact process with the root initially infected has two distinct phase transitions with $0<\lambda_1(\mathbb T_{d})<\lambda_2(\mathbb{T}_d)<\infty$, by a series of beautiful works by Pemantle \cite{p92} (for $d\ge 4$), Liggett \cite{l96} (for $d=3$), and Stacey \cite{s96}.  For a Galton-Watson tree $\T$ with offspring distribution $\xi$, it is not difficult to see that $\lambda_1(\T)$ and $\lambda_2(\T)$ are constants depending only on $\xi$, conditioned on $|\T|=\infty$. Huang and Durrett \cite{hd18} proved that on  $\T\sim \GW(\xi)$ with the root initially infected, $\lambda_2(\T)=0$ if the offspring distribution $\xi$ is subexponential, i.e., $\E e^{c\xi}=\infty$ for all $c>0$. So in this case, there is only the strong survival phase.

 By contrast, if the offspring distribution $\xi$  has an exponential tail, i.e., $\E e^{c\xi} <\infty$ for some $c>0$, Bhamidi, Nam and the authors \cite{BNNS} showed that there is an extinction phase: $\lambda_1(\T)\geq \lambda_0(\xi)$  for some constant $\lambda_0(\xi)>0$.  In \cite{NNS}, Nam et al. established the first-order asymptotics on $\lambda_1^{\textsc{gw}}(\xi)$ for $\xi$ sufficiently well \textit{concentrated} around its mean, in particular the case of Poisson degrees. The asymptotics turns out to have the same form as for regular trees.

 Aside from infinite graphs, there has been considerable work studying the phase transitions of survival times on large finite graphs. Stacey \cite{s01} and Cranston-Mountford-Mourrat-Valesin \cite{cmmv14} studied the contact process on the $d$-ary tree $\mathbb{T}_d^h$ of depth $h$ starting from the all-infected state, and their results show that the survival time $T_h$, as $h\rightarrow \infty$, satisfies (i) $T_h/h \rightarrow \gamma_1$ in probability if $\lambda < \lambda_2(\mathbb{T}_d)$; (ii) $|\mathbb{T}_d^h|^{-1}{\log \E T_h} \rightarrow \gamma_2$ in probability and $T_h/\E T_h \overset{d}{\rightarrow} \Exp(1)$ if $\lambda >\lambda_2(\mathbb{T}_d)$, where $\gamma_1,\gamma_2$ are constants depending on $d,\lambda$. In \cite{dl88, ds88, m93}, similar results were established for the case of the lattice cube $\{1,\ldots,n \}^d$.

 Recently in work of  Mourrat-Valesin~\cite{mv16} and Lalley-Su~\cite{ls17}, it was shown that for any $d\geq 3$, the contact process on the random $d$-regular graph, whose initial configuration is the all-infected state, exhibits the following phase transition:

 \begin{itemize}
 	\item (\textit{Short survival}) For $\lambda <\lambda_1(\mathbb{T}_d)$, the process survives for $O(\log n)$-time \textsf{whp}.
 	
 	\item (\textit{Long survival}) For $\lambda >\lambda_1(\mathbb{T}_d)$, the process survives for $e^{\Theta(n)}$-time \textsf{whp}.
 \end{itemize}

 For related results, we refer to  \cite{cs15, cd09, ls17, ms16, mmvy16, mvy13, mv16, sv2017} and the references therein.

Let us move on to finite graphs that are non-regular. We focus on the contact process on random graphs with a given degree distribution. Let $\mu$ be a degree distribution and  $G_n\sim \mathcal{G}(n,\mu)$ be the configuration model with degree distribution $\mu$, assuming that there exists a unique giant component (for details, see Section \ref{subsubsec:rgprelim}). Consider the contact process on $G_n$ where all vertices are initially infected. In \cite{BNNS}, it was shown that if  $\E_{D\sim \mu} e^{cD}<\infty$ {\footnote{Here, $\E_{D\sim \mu} f(D)$ denotes the expectation of $f(D)$ where $D$ is distributed according to $\mu$.}} for some constant $c>0$, then there exist constants $0<\underline{\lambda}(\mu) \le\overline{\lambda}(\mu) <\infty$ such that the survival time $T_{\lambda,n}$  of the process, which is the first time the process reaches the all-healthy state, satisfies the following:
 \begin{enumerate}
 	\item [(1)] For all $\lambda <\underline{\lambda}$,  $T_{\lambda,n}\leq n^{1+ o(1)}$ \textsf{whp};
 	
 	\item [(2)] For all $\lambda> \overline{\lambda}$, $T_{\lambda,n}\geq e^{\Theta(n)}$ \textsf{whp}.
 \end{enumerate}

 Based on this result, the short- and long-survival thresholds $\lambda_c^-(\mu)$, $\lambda_c^+(\mu)$ are defined as 
 \begin{equation*}
 \begin{split}
 \lambda_c^-(\mu) &:= \lim_{\alpha\to \infty}  \sup \Big\{\lambda: \;\lim_{n\to \infty}\P (T_{\lambda,n} \leq n^\alpha) = 1 \Big\};\\
 \lambda_c^+(\mu) &:= \lim_{\beta \to 0}
 \;\inf \left\{\lambda: \;\; \lim_{n\rightarrow \infty}\P(T_{\lambda,n} \geq e^{\beta n} )=1 \right\}.
 \end{split}
 \end{equation*}

 The first-order asymptotics for these thresholds are achieved \cite[Theorem 2]{NNS}:
 \begin{equation}
 \lambda_c^-(\mu) \sim \frac{1}{\E_{D\sim \tilde \mu} D}\quad\text{and}\quad \lambda_c^+(\mu) \sim \frac{1}{\E_{D\sim \tilde \mu} D},\label{eq:critical}
 \end{equation}
as $\E_{D\sim \tilde \mu} D$ tends to infinity where $\tilde \mu$ is the size-biased distribution of $\mu$ (see \eqref{eq:def:sizebiased} for the precise definition).

 On the other hand, if $\mu$ has a heavier than exponential tail, the story is quite simple (\cite{BNNS}): \textsf{whp} there is no short survival phase and so, $\lambda_c^-(\mu) = \lambda_c^+(\mu) = 0$. Therefore, we shall only restrict to measures with exponential tails in this paper. For these measures,
while it is now known that the survival time for $\la>\lambda_c^+(\mu)$ is exponential $e^{\Theta(n)}$, little is known about the survival time for small $\la$ except that it is of order $n^{O(1)}$.

\subsection{Main results} As promised, we now present generalizations of Theorem \ref{thm:ER} for random graphs with given degree distributions. In the following, we state the lower and upper bounds for the survival time separately because the assumptions needed for the lower bound are more relaxed.
		\begin{theorem}[Lower bound]\label{thm:lower:gen}
	 Let $\mu$ be any probability measure satisfying $\E_{D\sim\mu} D(D-2) >0$. Let $G_n\sim \mathcal G(n, \mu)$. Let $\la>0$ be sufficiently small so that $\la<1/20$ and
		\begin{equation}\label{def:A:lam}
		A = \frac{10^6}{\la^{2}} \log \frac{1}{\la} \ge 16\E _{D\sim \mu} D.
		\end{equation}
		Consider the contact process on $G_n$ where all vertices are initially infected and infection rate $\la$. Then this process still survives at time $n^{-c /\log (\mu[A, \infty))}$ with probability $1 - o(1)$ where $c = 1/120$ and the implicit constant in the $o(1)$ may depend on $\la$ and $\mu$.
	\end{theorem}
When $\mu$ has a truly exponential tail $\mu[A, \infty)\approx \exp(-A^{1+o(1)})$ for $A$ sufficiently large, this theorem concludes that the survival time is $\ge n^{c\la^{2+o(1)}}$ with high probability.

Concerning the upper bound, we achieve the same power for the survival time.
\begin{theorem} [Upper bound]\label{thm:uper}
		Let $\mu$ be a distribution satisfying $\E_{D\sim\mu} D(D-2) >0$ and
		\begin{equation}\label{cond:mu}
		 \E_{D\sim \mu} e^{10 D}<\infty.
		\end{equation}
		There exists a constant $C=C(\mu)>0$ such that for all $\la<1/C$, the contact process on $G_n\sim \mathcal G(n, \mu)$ starting at all vertices infected and infection rate $\lambda$ dies before time $n^{C\la^{2}\log(\la^{-1})}$ with probability $1 - o (1)$ as $n\to \infty$ where the implicit constant may depend on $\la$ and $\mu$.
		\end{theorem}
\begin{remark}
	 The constant $10$ in \eqref{cond:mu} may not be optimal. It would be an interesting and challenging task to show a matching upper and lower bound, taking the $\log(\la^{-1})$ factor into account. 
\end{remark}

\subsection{Main ideas} \label{sec:idea}

In the first part of the paper, we prove the lower bound in Theorem \ref{thm:lower:gen}.
It is well-known that the contact process survives for a long time on dense sub-graphs.  The simplest such example is a large star graph, a vertex with $k$ neighbors.  It is easy to verify that the expected survival length behaves like $e^{ck\lambda^2}$ (see Lemma \ref{lm:star}).  Since the largest degree in an Erd\H{o}s-R\'enyi random graph is asymptotically $\frac{\log n}{\log \log n}$, this directly gives a lower bound of $\exp(\frac{c\log n}{\log \log n})$ which far exceeds the $\log n$ survival time in the random regular graph.  However, these highest-degree vertices are not the optimal local structure.  Instead we find that for paths of vertices of degree $10^6 \la^{-2} \log \tfrac{1}{\la}$, the contact process behaves effectively like a supercritical one-dimensional contact process.  This has an infection time that is exponential in the length of the path and we show that with high probability there are paths of logarithmic length which gives the bound in Theorem \ref{thm:lower:gen}.

The second, and more challenging, part is to prove the matching upper bound of Theorem \ref{thm:uper}. As in~\cite{BNNS,NNS}, we first control the contact process on Galton-Watson Branching process trees which are the local weak limits of the random graphs.  In~\cite{BNNS}, we developed a recursive scheme to show that the expected survival time of an infection started from the root is constant, independent on the depth of the tree.  This was further refined in~\cite{NNS}, giving tail bounds on the conditional expected survival time given the tree as well as for the expected number of infections that reach the leaves.  These tail bounds, however, would only translate into an upper bound on the expected survival time on the random graph of $n^{\lambda^{1+o(1)}}$.  To explain where this analysis was deficient, it effectively coupled the process with a modifed process where the root of the tree must remain infected until all its subtrees became uninfected.  The advantage of this coupling is that the processes in the subtrees become independent.  However, if the degree of the root is $D$ then in the modified process, it will take time $\exp(c\la D)$ rather than $\exp\left (c\la^{2} D\right )$ until all the children of the root become uninfected.  What this analysis misses is that each infected child has only probability $\lambda$ of sending an infection back to its parent.

To give a sharp analysis, we introduce a modified dynamics that stochastically dominates the original where particles are strongly infected if they will send an infection to a neighbour before healing and weakly infected otherwise.  Weak infections can then be effectively ignored.  A vertex $u$ sends a strong infection to a neighbour $v$ with rate $\frac{\lambda^2 d_v}{1+\lambda d_v}$ which is roughly $C\lambda^{2}$ when the degree of $v$ is constant.  With this formulation, we get sharper bounds on the tail of the expected survival time and number of particles to hit the leaves.

To complete the proof, we use this analysis to show by a union bound that for any vertex $v$, starting with $v$ infected, the expected number of infections to exit a ball of radius $C\la^{2}\log (\la^{-1})\log n$ is $o(1)$.  When infections leave the ball, we analyse them as new contact processes started from the exit point and stochastically dominate this process by a subcritical branching process which with high probability terminates in $n^{C\la^{2}\log(\la^{-1})}$ steps.  Furthermore, we again use our recursive bounds to show that the maximum expected survival time on balls of radius $C\la^{2}\log (\la^{-1})\log n$ in the graph is at most $n^{C\la^{2}\log(\la^{-1})}$ with high probability.  Combining these results establishes Theorem~\ref{thm:uper}.

\subsection{Notation} The $\mathcal{G}(n,\mu)$ denotes the configuration model with degree distribution $\mu$ while  $\mathcal G(n, d/n)$ denotes the Erd\H{o}s-R\'enyi random graph.

For a tree $T$ and a depth $l$, we denote the set of vertices of $T$ at depth $l$ by $T_{l}$. We use $T_{\le l}$ to denote the set of vertices of depth at most $l$. In particular, $\T_{\le l}\sim\gw(\xi)_{\le l}$ denotes the Galton-Watson tree generated up to depth $l$, while the infinite Galton-Watson tree is denoted by $\T\sim\gw(\xi)$.

Throughout the paper, we often work with the contact process defined on a (fixed) graph generated at random.  To distinguish between the two randomness of different nature, we introduce the following notations:
\begin{itemize}
	\item $\Pcp$ and $\E_{\textsc{cp}}$ denote the probability and the expectation, respectively, with respect to the randomness from  contact processes.
	
	\item $\Pgw$ and $\P_\textsc{rg}$ denote the probability   with respect to the randomness from the underlying graph, when the graph is a Galton-Watson tree and  a random graph $\mathcal{G}(n,\mu)$, respectively. We write $\E_\textsc{gw}$ and $\E_\textsc{rg}$ similarly for expectations.
	
	\item $\P$ and $\E$ denote the probability and expectation, respectively, with respect to the combined randomness over both the process and the graph. That is, for instance, $\E[\cdot] = \E_\textsc{gw} [\E_\textsc{cp} [\cdot] ]$, if the underlying graph is a Galton-Watson tree.
\end{itemize}

\subsection{Organization} In Section \ref{sec:prelm}, we include the preliminaries for the graphical construction of the contact process (Section \ref{sec:graphical}), the models of random graphs (Section \ref{subsubsec:rgprelim}). Several useful results on the contact process on star graphs are presented in Section \ref{sec:star}. We prove Theorem \ref{thm:lower:gen} on the lower bound in Section \ref{sec:proof:lower}. In Section \ref{sec:upper:tree}, we establish the analog of Theorem \ref{thm:uper} for trees. And finally, we prove Theorem \ref{thm:uper} in Section \ref{sec:upper}.

\section{Preliminaries}\label{sec:prelm}
\subsection{Graphical representation of contact processes}\label{sec:graphical}

A standard construction of the contact processes uses the  \textit{graphical representation}, see for example  \cite[Chapter 3, section 6]{liggett:ips}. The idea is to record the infections and recoveries of the contact process on a graph $G$ on the space-time domain $G \times \mathbb{R}_+$. Define i.i.d. Poisson processes $\{N_v(t)\}_{v\in V}$ with rate $1$ (that represent the healing clock rings at vertex $v$) and i.i.d. Poisson processes $\{N_{\vec{uv}}(t) \}_{\vec{uv}\in \overrightarrow{E}}$ with rate $\lambda$ (that represent the clock rings for the infection from $u$ to $v$), where we consider $\overrightarrow{E}=\{\vec{uv},\;\vec{vu}:(uv)\in E \}$ to be the set of directed edges. Further, we let $\{N_v(t)\}_{v\in V}$ and $\{N_{\vec{uv}}(t) \}_{\vec{uv}\in \overrightarrow{E}}$ be mutually independent. Then the graphical representation is defined as follows:
\begin{enumerate}
	\item Initially, we have the \textit{empty} domain $V \times \mathbb{R}_+$.
	
	\item For each $v\in V$, mark  $\times$ at the point $(v,t)$, at each event time $t$ of $N_v(\cdot)$.
	
	\item For each $\vec{uv}\in \overrightarrow{E}$, add an arrow from $(u,t)$ to $(v,t)$, at each event time $t$ of $N_{\vec{uv}}(\cdot)$.
\end{enumerate}

\begin{figure}[H]
	\centering
	\begin{tikzpicture}[thick,scale=1.1, every node/.style={transform shape}]
	\foreach \x in {0,...,4}{
		\draw[black] (\x,0) -- (\x,3.8);
	}
	\draw[black] (0,0)--(4,0);
	\draw[dashed] (0,3.8)--(4,3.8);
	\draw (0,.8) node[cross=2.2pt,black]{};
	\draw (0,3.2) node[cross=2.2pt,black]{};
	\draw (1,2.4) node[cross=2.2pt,black]{};
	\draw (2,.4) node[cross=2.2pt,black]{};
	\draw (3,1.36) node[cross=2.2pt,black]{};
	\draw (4,.88) node[cross=2.2pt,black]{};
	\draw (4,2.8) node[cross=2.2pt,black]{};
	
	\draw[->,line width=.6mm,black] (0,2.96)--(.57,2.96);
	\draw[black] (0,2.96)--(1,2.96);
	\draw[->,line width=.6mm,black] (1,1.44)--(.43,1.44);
	\draw[black] (1,1.44)--(0,1.44);
	\draw[->,line width=.6mm,black] (1,2.08)--(1.57,2.08);
	\draw[black] (1,2.08)--(2,2.08);
	\draw[->,black] (2,1.2)--(1.43,1.2);
	\draw[black] (1,1.2)--(2,1.2);		
	\draw[->,black] (2,1.84)--(2.57,1.84);
	\draw[black] (2,1.84)--(3,1.84);
	\draw[->,black] (3,2.72)--(2.43,2.72);
	\draw[black] (2,2.72)--(3,2.72);		
	\draw[->,line width=.6mm,black] (3,.56)--(3.57,.56);
	\draw[black] (3,.56)--(4,.56);
	\draw[->,black] (3,3.28)--(3.57,3.28);
	\draw[black] (3,3.28)--(4,3.28);
	\draw[->,black] (4,2.16)--(3.43,2.16);
	\draw[black] (3,2.16)--(4,2.16);			
	\node at (4.7,0.1) {$t=0$};
	\node at (4.7,3.8) {$t=s$};
	\node at (0,-.3) {$1$};
	\node at (1,-.3) {$2$};
	\node at (2,-.3) {$3$};
	\node at (3,-.3) {$4$};
	\node at (4,-.3) {$5$};
	\draw [line width=.07cm, blue] (0,0) -- (0,.8);
	\draw [line width=.07cm, blue] (1,0) -- (1,2.4);	
	\draw [line width=.07cm, blue] (1,1.44)--(0,1.44) -- (0,3.2);
	\draw [line width=.07cm, blue] (0,2.96) -- (1,2.96)--(1,3.8);
	\draw [line width=.07cm, blue] (1,2.08) -- (2,2.08)--(2,3.8);	
	\draw [line width=.07cm, blue] (2,0) -- (2,.4);
	\draw [line width=.07cm, blue] (3,1.36) -- (3,0);
	\draw [line width=.07cm, blue] (4,0) -- (4,.88);
	\draw [line width=.07cm, blue] (3,.56) -- (4,.56);			
	
	\end{tikzpicture}
	\caption{A realization of the contact process on the interval $V=\{1,\ldots,5\}$, with initial condition $X_0=\one_V$. The blue lines describe the spread of infection. We see that $X_s = \one_{\{2,3\}}$.} \label{fig1}
\end{figure}

\noindent This gives a geometric picture of the contact process, and further provides a coupling of the processes over all possible initial states. Figure \ref{fig1} tells us how to interpret the infections at time $t$ based on this graphical representation. The following are two simple yet important observations using this construction (ref. \cite[Chapter 3, section 6]{liggett:ips}).

\begin{lemma} \label{lem:graph rep1}
	Suppose that we have the aforementioned coupling among the contact processes on a graph $G$. Let $T_v$ be the survival time of the contact process on $G$ with $v$ initially infected. Let $T_G$ be the survival time of the contact process on $G$ with all vertices initially infected.  Then we have $T_G = \max \{T_v:v\in G \}$.
\end{lemma}

\begin{lemma} \label{lem:graph rep2}
	For a given graph $G=(V,E)$ and any $A\subset V$, let $(X_t)$ be the contact process on $G$ with $A$ initially infected. Consider any (random) subset $\mathcal{I}$ of $ \mathbb{R}_+$, and  define $(X_t')$ to be the coupled process of $(X_t)$ that has the same initial state, infections and recoveries, except that the recoveries at a fixed vertex $v$ are ignored at times $t\in \mathcal{I}$. Then for any $t\geq 0$, we have $X_t \leq X_t'$, i.e., $X_t(v) \leq  X_t'(v)$ for all $v$.
\end{lemma}

\subsection{Random graphs and their limiting structure}\label{subsubsec:rgprelim}

Let $\mu$ be a probability distribution on $\mathbb{N}$. The random graph $G_n \sim \mathcal{G}(n,\mu)$ with degree distribution $\mu$ is defined as follows:

\begin{itemize}
	\item Let $d_i \sim $ i.i.d.$\;\mu$ for $i=1,\ldots, n$, conditioned on $\sum_{i=1}^n d_i \equiv 0$ mod $2$. The numbers $d_i$ refer to the number of \textit{half-edges} attached to vertex $i$.
	
	\item Generate the graph $G_n$ by pairing all half-edges uniformly at random.
\end{itemize}
The resulting graph $G_n$ is also called the \textit{configuration model}. One may also be interested in the \textit{uniform model} $G^{\textsf{u}}_n \sim \mathcal{G}^{\textsf{u}}(n,\mu)$, which picks a uniformly random simple graph among all simple graphs with degree sequence $\{d_i\}_{i\in[n]} \sim $ i.i.d.$\;\mu$. It is well-known that if $\mu$ has a finite second moment, then the two laws $\mathcal{G}(n,\mu)$ and $\mathcal{G}^{\textsf{u}}(n,\mu)$ are \textit{contiguous}, in the sense that for any subset $A_n$ of graphs with $n$ vertices,
\begin{equation*}
\P_{G_n\sim \mathcal{G}(n,\mu) } \left(G_n \in A_n \right) \rightarrow 0
\quad
\textnormal{implies}
\quad
\P_{G^{\textsf{u}}_n\sim \mathcal{G}^{\textsf{u}}(n,\mu) } \left(G^{\textsf{u}}_n \in A_n \right) \rightarrow 0.
\end{equation*}
For details, we refer the reader to Chapter 7 of \cite{vanderhofstad17} or to \cite{j09}. We remark that when $\mu=\textnormal{Pois}(d)$, the random graph $G_n\sim \mathcal{G}(n,\mu)$ is contiguous to the Erd\H{o}s-R\'enyi random graph $G_n^{\textsc{er}}\sim \mathcal{G}_{\textsc{er}} (n,d/n)$ as shown in \cite{k06}, Theorem 1.1.

Furthermore, it is also well-known that the random graph $G_n \sim \mathcal{G}(n,\mu)$ is \textit{locally tree-like}, and the local neighborhoods converge \textit{locally weakly} to Galton-Watson trees. To explain this precisely, let us denote the law of Galton-Watson tree with offspring distribution $\mu$ by $\gw(\mu)$, and let $\gw(\mu)_{\le l}$ be the law of $\gw(\mu)$ truncated at depth $l$, that is, the vertices with distance $>l$ from the root are removed. Further, let $\widetilde \mu$ denote the \textit{size-biased} distribution of $\mu$ defined by
\begin{equation}\label{eq:def:sizebiased}
\widetilde{\mu}(k-1) := \frac{k\mu(k)}{\sum_{k'=1}^\infty k'\mu(k') }, \quad k=1,2,\ldots . \footnote{This differs from a possibly  more common definition of size-biased distribution where the left-hand side would be the mass at $k$.}
\end{equation}
Note that if $\mu=$Pois$(d)$, then $\widetilde{\mu}= \mu$. Lastly, define $\gw(\mu, \widetilde{\mu})_{\le l}$ to be the Galton-Watson process truncated at depth $l$, such that the root has offspring distribution $\mu$ while all other vertices have offspring distribution $\widetilde{\mu}$. Then the following lemma shows the convergence of local neighborhoods of $G_n$.

\begin{lemma}[\cite{DemboMontanari2010}, Section 2.1]\label{lem:lwc} Suppose that $\mu$ has a finite mean. Let $l>0$ and let $v$ denote the vertex in $G_n\sim \mathcal{G}(n,\mu)$ chosen uniformly at random. Then for any rooted tree $(T, x )$ of depth $l$, we have
	\begin{equation*}
	\lim_{n\rightarrow \infty}\P((N(v,l),v) \cong (T,x)  ) =
	\P_{(\mathcal{T},\rho )\sim \gw(\mu,\widetilde{\mu})_{\le l}} ((\mathcal{T},\rho) \cong (T,x)  ),
	\end{equation*}
	where $N(v,l)$ is the $l$-neighborhood of $v$ in $G_n$ and $\cong$ denotes the  isomorphism of rooted graphs. We say that $G_n$ converges locally weakly to $\gw(\mu,\widetilde{\mu})$.
\end{lemma}

\noindent We remark that the same holds for a fixed vertex $v\in G_n$. Moreover, we also stress that the condition for $\gw(\mu, \widetilde{\mu})$ to be supercritical is equivalent to the condition for $\mathcal{G}(n,\mu)$ to have the unique giant component \textsf{whp} (see e.g., \cite{mr95}, or \cite{durrett:rgd}, Section 3 for details), which can be addressed as
\begin{equation}\label{eq:condition on mu}
\E_{D\sim\mu} D(D-2) >0.
\end{equation}

\section{Contact process on star graphs}\label{sec:star}
In this section, we discuss several results about the contact process on star graphs that will be useful later. For a positive integer $k$, let $S_{k}$ be a star graph with a root, which we usually denote by $\rho$, and $k$ leaves. Our goal is to provide a lower bound for the survival time of the contact process on $S_{k}$ which we state in Lemma \ref{lm:star}. Before doing that, we shall need some preparation. Firstly, we show that the root infects many leaves before its recovery.

\begin{lemma}\label{lm:star:ini}
Consider the contact process $(X_t)$ with infection rate $\la<1$ on the star graph $S_k$. Let $\ell = \la k$. Assume that the root is infected at time $0$. Let $C\le \frac{\ell}{32\log \ell}$ be a positive number. Then the probability that $\rho$ infects at least $C\log \ell$ leaves who remain infected during time $[\frac{16C\log \ell}{\ell}, 1]\supset [1/2, 1]$ is at least $1 - \frac{16C\log \ell}{\ell} - 4\ell^{-C/4}$.
\end{lemma}

\begin{proof}
 Let $s = \frac{16C\log \ell}{\ell}$.
 The probability that the root is still infected  until time $s$ is $\P(\Pois(s) = 0) = e^{-s}\ge 1 - s = 1 - \frac{16C\log \ell}{\ell}$.

 The probability that  there are at least $8C\log \ell$ infection rings from $\rho$ to its leaves in time $[0, s]$ is
 $$\P(\Pois(\ell s)\ge 8C\log \ell) = \P(\Pois(16C\log \ell)\ge 8C\log \ell) \ge 1 - 2\exp(-8C\log \ell/6) \ge 1 - 2\ell^{-C},$$
 where we use the following standard tail bound for Poisson distribution (see, for instance, \cite[Chapter 2]{massartconcentrationbook}):
 \begin{equation}\label{eq:poisson}
 \P(|\Pois(\La) - \La|\ge x)\le 2\exp\left (-\frac{x^{2}}{2(\La+x)}\right ), \quad \forall x>0.
 \end{equation}

 Under this event, letting $M$ be the integer $M=4[C\log \ell]$, the probability that these rings correspond to at least $M$ distinct leaves is at least
 $$1-\frac{{k\choose M} M^{2M}}{k^{2M}}\ge 1-\left (\frac{ek M^{2} }{Mk^{2}}\right )^{M}=1-\left (\frac{4e C\log \ell}{k}\right )^{4C\log \ell}\ge 1- \exp(-4C\log \ell) = 1-\ell^{-4C},$$
 where we used $C\le k/(32\log \ell)$.

For each leaf infected before time $s$, the probability that it remains infected until time $1$ is at least $\P(\Pois(1) = 0) = e^{-1}$. Thus, given that there are $\ge M$ leaves infected before time $s$, at least $M/4$ of them remain infected until time $1$ with probability at least $1 - \exp(-M/16) = 1-\ell^{-C/4}$ where
we used the Chernoff inequality
\begin{equation}\label{eq:chernoff}
\P(\text{Bin}(M, q) \le Mq\delta)\le \exp\left (-\frac{(1-\delta)^{2}Mq}{2}\right ), \quad \forall M\in \mathbb N, \text{ } q, \delta\in (0, 1).
\end{equation}

Combine all of these events and the union bound, we obtain the desired statement.
 \end{proof}
The next two lemmas show that once many leaves are infected, they keep the root well-infected which then keeps the number of infected leaves from decreasing.
\begin{lemma}\label{lm:star:rho}
	 Consider the contact process $(X_t)$ with infection rate $\la<1$ on the star graph $S_k$. Let $M$ be an integer satisfying $\frac{8}{\la}\le M\le k$. Assume that at least $M$ leaves are infected at time $t$ and remain infected until time $t+1/2$. The probability that $\rho$ is infected for at least half of the time interval $[t, t+1/2]$ is at least $1 - 2 \exp(-\la M/32)$.
\end{lemma}
\begin{proof}
	Conditioned on the event that the $M$ leaves remain infected during $I:=[t, t+1/2]$, it holds that $\rho$ is infected with rate at least $\la M$. Let $Y\sim \Pois(1/2)$ be the number of healing clocks at $\rho$ in $I$. The probability that $\rho$ is healthy for at least $1/2$ of $I$ is at most
	\begin{eqnarray}
	\P\left (\sum_{i=1}^{Y} \eta_i \ge 1/4\right ),\nonumber
	\end{eqnarray}
	where $\eta_i\sim \exp(\la M)$ is the waiting time until the next infection after the $i$-th healing clock. We note that these $\eta_i$'s are iid. We have
	\begin{eqnarray*}
	\P\left (\sum_{i=1}^{Y} \eta_i \ge 1/4\right )&\le& \P\left (Y\ge \la M/8\right )  +\P\left (\sum_{i=1}^{\la M/8} \eta_i\ge 1/4\right )\\
	&=& \P(Y\ge \la M/8)  +\P\left (\sum_{i=1}^{\la M/8} \eta'_i\ge \la M/4\right ) \quad \text{where $\eta_i' = \eta_i \la M\sim \exp(1)$}\\
	&\le&  \exp(-\la M/32) +\P\left (\sum_{i=1}^{\la M/8} \eta'_i\ge \la M/4\right ) \quad \text{by \eqref{eq:poisson}}\\
		&\le&  \exp(-\la M/32) +\exp(-\la M/8)\E \exp\left (\sum_{i=1}^{\la M/8} \eta'_i/2\right )\\
			&\le&  \exp(-\la M/32) +2^{\la M/8}\exp(-\la M/8) \le 2 \exp(-\la M/32) .
	\end{eqnarray*}
This completes the proof.
\end{proof}

\begin{lemma}\label{lm:star:iteration}
	  Consider the contact process $(X_t)$ with infection rate $\la<1$ on the star graph $S_k$. Let $M$ be an integer satisfying $\frac{32}{\la}\le M\le \frac{\la k}{100}$. Assume that at least $M$ leaves are infected at time $t$. Then the probability that at least $\frac{\la k}{100}$ leaves are infected at time $t+1$ is at least $1 - 3\exp(-\la M/128)$.
\end{lemma}
\begin{proof}
	Let $m$ be the number of leaves infected at time $t$. We have $m\ge M$. Consider two cases.
	
	\begin{itemize}
		\item Case 1: $m\ge k/2$. Each of these infected leaves has a probability of $\P(\Pois(1) = 0) = e^{-1}$ to remain infected until time $t+1$. Thus, by Chernoff inequality, at least $m/4\ge \frac{\la k}{100}$ of them remain infected until time $t+1$ with probability at least $1 - \exp(-m/12) \ge 1 - 3\exp(-\la M/128)$.
		
		\item Case 2: $M\le m \le k/2$. By the same argument as in Case 1, with probability at least $1 - \exp(-m/16) $, at least $M/4$ of these $m$ infected vertices remain infected until time $t+1/2$. Under this event, by Lemma \ref{lm:star:rho}, with probability at least $1 - 2 \exp(-\la M/128)$, $\rho$ is infected for at least half of the time interval $[t, t+1/2]$. Thus, each of the remaining $k-m\ge k/2$ leaves receives an infection from $\rho$ during $[t, t+1/2]$ with probability at least
		$$\P(\Pois(\la/4)\ge 1) = 1- e^{-\la/4}\ge \la/8.$$
		Once they receive an infection, the probability that they remain infected until time $t+1$ is at least $e^{-1}$. Thus, the probability that at least $ \frac{\la k }{100}$ vertices are infected at time $t+1$ is at least $1 - \exp(-\la k/100)$.
	\end{itemize}
This gives the desired statement for both cases.
\end{proof}

Combining these lemmas, we obtain the following lower bound on the survival time of the contact process on star graphs.
\begin{lemma}\label{lm:star}
	 Consider the contact process $(X_t)$ with infection rate $\la$ on the star graph $S_k$. Assume that $\rho$ is infected at time 0. Then the contact process still survives at time $e^{\la^{2}k /10^5}$ with probability at least $1 - \frac{200 \log (\la k)}{\la^{2}k} - 3e^{-\la^{2}k /10^5}$ for all $k\ge \la^{-2}10^{5}$.
\end{lemma}
\begin{proof}
	Let $C$ be such that $\frac{8}{\la}\le C\le \frac{\la k}{32\log (\la k)}$ to be chosen (these inequalities are so that the bounds in the hypotheses of the above lemmas are satisfied). We consider the event $\mathfrak A$ that consists of the following events
	\begin{enumerate}
		\item [$\mathfrak A_1$.] The root $\rho$ sends its infection to at least $C\log (\la k)$ vertices before time $1/2$ who remain infected until time 1,
		\item [$\mathfrak A_2$.] At least $\frac{\la k}{100}$ vertices are infected at time $2$,
		\item [$\mathfrak A_h$.] At least $\frac{\la k}{100}$ vertices are infected at time $h$ for $h=3, 4, \dots, \lceil e^{\la^{2}k/10^5}\rceil$.
	\end{enumerate}
Under this event, it is clear that the contact process survives until time $e^{\la^{2}k /10^5}$. By Lemma \ref{lm:star:ini}, $\mathfrak A_1$ happens with probability at least $1 -  \frac{16C\log (\la k)}{\la k} - 4(\la k)^{-C/4}$.

By Lemma \ref{lm:star:iteration}, conditioned on $\mathfrak A_1$, the event $\mathfrak A_2$ happens with probability at least
$$1 - 3 \exp(-\la C\log(\la k)/128) = 1 - 3 (\la k)^{-\la C/128}.$$

Similarly, by Lemma \ref{lm:star:iteration}, each $h\ge 3$, conditioned on $\mathfrak A_{h-1}$, the event $\mathfrak A_h$ happens with probability at least
$$1 - 3 \exp(-\la^{2}k/(128\cdot 100)).$$

By the union bound, we obtain
\begin{eqnarray*}
1 - \P(\mathfrak A) &\le& \frac{16C\log (\la k)}{\la k} + 4(\la k)^{-C/4}+ 3 (\la k)^{-\la C/128}+3 \exp(-\la^{2}k/(128\cdot 100))\times \exp(\la^{2}k/10^5)\\
&\le& \frac{16C\log (\la k)}{\la k} + 7(\la k)^{-C/4} +3 \exp(-\la^{2}k/10^5)\\
&\le& \frac{23C\log (\la k)}{\la k} +3 \exp(-\la^{2}k/10^5)\quad\text{since  $C\ge 4$}\\
&\le& \frac{200\log (\la k)}{\la^2 k} +3 \exp(-\la^{2}k/10^5)\quad\text{by choosing $C=8/\la$.}
\end{eqnarray*}
 This completes the proof.
\end{proof}

	\section{Proof of the lower bound: Theorem \ref{thm:lower:gen}}\label{sec:proof:lower}
 		Our strategy is to show that in $G_n$, there exists a long path of high-degree vertices which we think about as a long path on which each vertex is attached to a big star (particularly, an $(A-2)$-star). For each of these vertices, Lemma \ref{lm:star} shows that the contact process on their attached star survives for a long time. So, each infected star has ample time to pass the infection to other stars on the path before being healed.
		Theorem \ref{thm:lower:gen} thus follows immediately from the following propositions. The first proposition asserts the existence of a long path of high-degree vertices.
		\begin{proposition}[Existence of long path]\label{prop:path} Let $\mu$ be any probability measure satisfying $\E_{D\sim\mu} D(D-2) >0$. Let $G_n\sim \mathcal G(n, \mu)$. Let $A>16 \E_{D\sim \mu} D$ and
			\begin{equation}\label{def:L}
			L = \frac{0.9\log n}{-\log \mu[A, \infty)}.
			\end{equation}
			With high probability, there exists a path in $G_n$ of length $L$ in which every vertex has at least $A-2$ neighbors outside of the path and these $L(A-2)$ neighbors are all distinct.
		\end{proposition}
	\begin{figure}[H]
		\begin{center}
			\begin{tikzpicture}[scale=0.5]
			\foreach \x in {0,1,2,3,4,5,6}{
				
			\draw[-] (2*\x, 0) to (2*\x+2, 0);
		}
	\foreach \x in {0,1,2,3,4,5,6,7}{
		\node[circle, fill, draw, inner sep=1pt]  at (2*\x,0) {};
		\draw[-] (2*\x, 0) to (2*\x, -1);
		\draw[-] (2*\x, 0) to (2*\x-0.2, -1);
		\draw[-] (2*\x, 0) to (2*\x+0.2, -1);
	}
 		
			\end{tikzpicture}
		\end{center}
		\caption{The path $\Gamma$}\label{fig:gm}
	\end{figure}
	Let $\Gamma$ be the union of a path of length $L$ together with $L$ disjoint $(A-2)$-stars attached to vertices on the path (see Figure \ref{fig:gm}). From now on, we simply refer to $\Gamma$ as {\it the long path}. The next proposition shows that the contact process on the long path survives for a long time.
		\begin{proposition}[Survival time on the long path]\label{prop:path:survival}
			Let $A = \frac{10^6}{\la^{2}} \log \frac{1}{\la}$. The contact process on $\Gamma$ starting with all vertices infected survives until time $1.01^{L} \ge n^{-c/\log \mu[A, \infty)}$ with high probability as $L\to \infty$, where $c = 1/120$.
		\end{proposition}
	In the rest of this section, we prove these two propositions.
\subsection{Existence of a long path}
 To prove Proposition \ref{prop:path}, we shall use the cut-off line algorithm introduced in \cite{k06} to find a uniform perfect matching of the half-edges in $G_n$.	

 \begin{figure}[H]
 	\centering
 	\begin{tikzpicture}[thick,scale=1, every node/.style={transform shape}]
 	
 	\node at (2,-0.4) {$v_3$};
 	\node at (4,-0.4) {$\dots$};
 	\draw[black] (0,0)--(8,0);
 	
 	\node at (1,-0.4) {$v_2$};

 	\node at (8,-0.4) {$v_n$};
 	\node at (6,-0.4) {$\dots$};
 	\draw[red] (-1,3.2)--(9,3.2);
 	\draw (2, 0.5) node[black]{\textsf{o}};
 	\foreach \x in {0,1, 2, 4, 6, 8}{
 		\draw[black] (\x,0) -- (\x,4.8);
 	}
 	\draw[red, dashed] (-1,2.8)--(9,2.8);
 	\node at (10.4,2.8) {new cut-off line};
 	\draw (0, 3.2) node[black]{\textsf{o}};
 	
 	\draw (4, 4) node[black]{\textsf{o}};
 	
 	\draw (2,1.6) node[cross=2.2pt,black]{};
 	\draw (0, 0.5) node[cross=2.2pt,black]{}; 
 	\draw (0,2.2) node[cross=2.2pt,black]{};

 	\draw (4,1.36) node[cross=2.2pt,black]{};
 	\draw (4,2.8) node[cross=2.2pt,red]{};
 	\node at (10,3.2) {cut-off line}; 
 	\draw (6,0.3) node[cross=2.2pt,black]{};
 	\draw (1,2.4) node[cross=2.2pt,black]{};
 	
 	\draw (8, 3.7) node[black]{\textsf{o}};
 	
 	\draw (6,2.5) node[cross=2.2pt,black]{};
 	
 	\node at (0,-0.4) {$v_1$};
 	
 	\draw[->, blue] (5.8,1.6) -- (4.2,2.7);
 	
 	\draw (8, 3) node[black]{\textsf{o}};
 	\draw (8,.88) node[cross=2.2pt,black]{};
 	\draw (6,1.5) node[cross=2.2pt,blue]{};
 	\draw (6, 3.5) node[black]{\textsf{o}};
 	\end{tikzpicture}
 	\caption{The circles `o' represent matched half-edges and the crosses `$\times$' represent unmatched half-edges. The blue half-edge is chosen and matched to the red half-edge which is the highest unmatched half-edge. Then the cut-off line is moved to the new cut-off line (the dashed line).} \label{fig2}
 \end{figure}
 \begin{definition}[Cut-off line algorithm]\label{def:cut-off line}
	A perfect matching of the half-edges of $G_n$ is obtained through the following algorithm (see Figure \ref{fig2}).
	\begin{itemize}
		\item Each half-edge of a vertex $v$ is assigned with a height uniformly chosen in $[0, 1]$ and is placed on the line of vertex $v$.
		\item The cut-off line is initially set at height 1.
		\item Pick an unmatched half-edge independent of the heights of all unmatched half-edges and match it to the highest unmatched half-edge. Move the cut-off line to the height of the latter half-edge.
	\end{itemize}
\end{definition}
	\begin{proof} [Proof of Proposition \ref{prop:path}]
		We will use the cut-off line algorithm to explore the edges of $G_n$. Firstly, the degrees of vertices in $G_n$ are sampled and each half-edge of these vertices is assigned a height uniformly chosen in $[0, 1]$ and is placed on the line of the corresponding vertex $v$ as described in the cut-off line algorithm. We shall repeat the following greedy algorithm to find a long path of length $L$ defined in \eqref{def:L}. Let $U$ denote the set of visited vertices. We start with $U = \emptyset$.
		
		Let $\ep= 0.05$. We shall perform the algorithm in $n^{1-\ep}$ rounds and show that at least one of them will successfully find a long path. In the $k$-th round where $k=1, 2, \dots, n^{1-\ep}$, pick a vertex $v_{k, 1}\notin U$ of degree at least $A$ uniformly at random to be the first vertex on the long path for this round. Put $v_{k, 1}$ into $U$ and pick $A$ half-edges of $v_{k, 1}$ uniformly at random and match them with the highest unmatched half-edges one by one. Put these $A$ newly found neighbors of $v_{k, 1}$ into $U$.
		If any of these visited vertices coincide or have already been in $U$, the round fails and we move to the next round. Otherwise, let $v_{k, 2}$ be chosen uniformly at random among these $A$ neighbors of $v_{k, 1}$.
		If $\deg(v_{k, 2})< A$, the round fails and we move to the next round. Otherwise, $v_{k, 2}$ will be the second vertex on the long path for this round. We repeat the above process to find $v_{k, 3}, \dots, v_{k, L}$. So, either the round fails during the process or we have found the path $v_{k,1}, \dots, v_{k, L}$ each of which has $A-2$ neighbors all distinct.
		
		We now show that each round succeeds with a decent probability. To this end, we first provide some elementary bounds on the number of edges of $G_n$.
	By Chebyshev's inequality, \whp, the total number of half-edges in $G_n$ is in $[nd, 4nd]$ where $d = \E_{D\sim \mu} D$.
	Let $\delta$ be a sufficiently small constant depending only on $A$ and $\mu$. Let $\Delta$ be such that
	$$ \E_{D\sim \mu} D\one_{D\ge \Delta}\le \delta/4.$$
		
 For each vertex $v\in G_n$, consider the random variable $X_v := \deg_{G_n}(v)\one_{\deg_{G_n}(v) \ge \Delta}$. They are independent random variables with mean at most $\delta/4$ and variance bounded by the second moment of $\mu$. By Chebyshev's inequality, \textsf{whp}
 $$\sum_{v\in G_n} X_v\le \delta n/2.$$
   Applying Chernoff inequality to the bounded random variables $\bar X_v := \deg_{G_n}(v)\one_{\deg_{G_n}(v) < \Delta}$ we obtain that \textsf{whp},
 $$\sum _{v\in G_n} \bar X_{v} \in (1-\delta, 1+\delta) nd.$$
		Thus, \whp, the total degree in $G_n$ is
		\begin{equation}\label{eq:sum:deg}
		\sum _{v\in G_n}\deg_{G_n}(v) = \sum _{v\in G_n} (X_v + \bar X_{v}) \in ((1-\delta)nd, (1+2\delta) nd)\subset (nd/2, 2nd).
		\end{equation}
 		During each round, the number of visited vertices is at most $2AL$ and so during the algorithm, the number of matched half-edges is at most $2n^{1-\ep} AL\le nd/4$ for sufficiently large $n$.

 		Conditioned on these events, we now provide a lower bound for the probability that a round succeeds. Let $W$ be the set of vertices in $G_n$ with degree at least $A$. With high probability, $|W|\ge \frac{n\mu[A, \infty) }{2}$ and so its number of half-edges is at least $\frac{nA\mu[A, \infty)}{2}$. For any $i\in [L]$ and $k\in [n^{1-\ep}]$, the probability that $v_{k, i}\in W$ is the same as the probability that the highest unmatched half-edge belongs to a vertex in $W$. This happens with probability $\ge \frac{nA\mu[A, \infty)}{2\cdot 4nd} = \frac{A\mu[A, \infty)}{8d}$ because by \eqref{eq:sum:deg}, the total number of unmatched half-edges is at most $4nd$.

 		The total number of half-edges that the vertices visited during the algorithm is $\le \Delta n^{1-\ep} AL+\sum_{v\in G_n} X_v\le \Delta n^{1-\ep} AL+\delta n/2$. And so, at any stage of the algorithm, the probability that a vertex is revisited is $\le \frac{\Delta n^{1-\ep} AL+\delta n/2}{nd/4}$, where the denominator bounds from below the number of unmatched half-edges. Thus, for each $i\in [L]$ and $k\in [n^{1-\ep}]$, the probability that the exploration at $v_{k, i}$ succeeds is at least $\frac{A\mu[A, \infty)}{8d} - A\frac{\Delta n^{1-\ep} AL+\delta n/2}{nd/4}\ge \frac{A\mu[A, \infty)}{16d}$ for sufficiently large $n$. Here we recall that $\delta$ is a sufficiently small constant depending on $A$ and $\mu$.
 		
 		Thus, the probability that a round $k\in [n^{1-\ep}]$ succeeds is at least $\frac{A^{L}(\mu[A, \infty))^{L}}{(16d)^{L}}\ge (\mu[A, \infty))^{L} = n^{-0.9}\ge n^{-1+2\ep}$, where we used $A\ge 16d$. And so with high probability, there is at least one round that succeeds. Each successful round gives a desired long path.
	\end{proof}

	\subsection{Long survival on long paths}
	To prove Proposition \ref{prop:path:survival}, we couple the contact process on $\Gamma$ with a percolation process (Lemma \ref{lm:dominate:percolation}) and then show that this percolation survives for a long time (Lemma \ref{lm:percolation:survival}).
	
	Recall the definition of $A$ in \eqref{def:A:lam} that 	$A = \frac{10^6}{\la^{2}} \log \frac{1}{\la}$. By Lemma \ref{lm:star}, the contact process on an $(A-2)$-star with the root initially infected survives until time $a^{2}$ with
	  \begin{equation}\label{def:a}
		a = \lfloor e^{\la^{2}(A-2) /(2\cdot 10^5)} \rfloor
	\end{equation}
 with probability at least $1 - \frac{200 \log (\la (A-2))}{\la^{2}(A-2)} - 3e^{-\la^{2}(A-2) /10^5}\ge 0.99$ because $\la<1/20$. This allows us to couple the contact process on $\Gamma$ with a (bond) percolation process on the following directed graph $\mathcal L$ on the vertex set $V_{\L} = [L]\times \N$ with directed edges from $(i, t)$ to $(i, t+1)$ and $(i\pm 1, t+1)$ (as long as they belong to $V_{\L}$)  for all $(i, t)\in V_{\L}$ (see Figure \ref{fig:L}). The vertex $(i, t)$ on $\L$ represents the $i$-th star of $\Gamma$ at time $ta$. Having $(i, t)$ belong to an open path from an initial state corresponds to having at least $\la (A-2)/100$ vertices in the $i$-th star of $\Gamma$ infected at time $ta$. In the latter event, these vertices keep the root $v_i$ of the $i$-th star infected for a long time during the time interval $[ta, (t+1)a]$, giving it an ample chance to infect the $(i\pm 1)$-th stars by Lemma \ref{lm:star}.
\begin{figure}[H]
	 \begin{center}
	 	\begin{tikzpicture}[scale=0.5]
	 	\foreach \x in {0,1,2,3,4,5,6,7}
	 	\foreach \y in {0,1,2,3}
	 		\node[circle, fill, draw, inner sep=1pt]  at (2*\x,2*\y) {};
	 	\node at (-2.5,0) {t=0};
	 	\node at (-2.5,2) {t=a};
	 	\node at (-2.5,4) {t=2a};
	 	\node at (-2.5,6) {t=3a};

	 	\foreach \x in {1,2,3,4,5,6}
	 	\foreach \y in {0,1,2,3} {
	 		\draw[postaction={decorate,decoration={markings,mark=at position 0.3 with {\arrow{>}}}}] (2*\x,2*\y) -- (2*\x-2,2*\y+2);
	 		\draw[postaction={decorate,decoration={markings,mark=at position 0.3 with {\arrow{>}}}}]  (2*\x,2*\y) -- (2*\x,2*\y+2);
	 		\draw[postaction={decorate,decoration={markings,mark=at position 0.3 with {\arrow{>}}}}]  (2*\x,2*\y) -- (2*\x+2,2*\y+2);
	}
 	
 	\foreach \x in {0}
 	\foreach \y in {0,1,2,3} {
 		\draw[postaction={decorate,decoration={markings,mark=at position 0.3 with {\arrow{>}}}}]  (2*\x,2*\y) -- (2*\x,2*\y+2);
 		\draw[postaction={decorate,decoration={markings,mark=at position 0.3 with {\arrow{>}}}}]  (2*\x,2*\y) -- (2*\x+2,2*\y+2);
 	}
 
  	\foreach \x in {7}
  	\foreach \y in {0,1,2,3} {
  		\draw[postaction={decorate,decoration={markings,mark=at position 0.3 with {\arrow{>}}}}] (2*\x,2*\y) -- (2*\x-2,2*\y+2);
  		\draw[postaction={decorate,decoration={markings,mark=at position 0.3 with {\arrow{>}}}}]  (2*\x,2*\y) -- (2*\x,2*\y+2);
  	}
 
	 	 	\end{tikzpicture}
	 \end{center}
 	 	\caption{The graph $\L$}\label{fig:L}
\end{figure}

 	\begin{lemma}[Coupling with Percolation]\label{lm:dominate:percolation} 
		Consider the bond percolation on the graph $\L$ where each of the directed edges is open with probability $.99$, independently.
		There exists a coupling of the contact process on $\Gamma$ with all vertices initially infected and the percolation such that if a vertex $(i, t)$ is connected with some vertex $(j, 0)$ by a directed open path in the percolation then there are at least $\la (A-2)/100$ leaves in the $i$-th star of $\Gamma$ that are infected at time $ta$ where $a$ is as in \eqref{def:a}.
	\end{lemma}
	\begin{proof}[Proof of Lemma \ref{lm:dominate:percolation}]
		For each $i\in [L]$, let $\A^{t}_i$ be the event that there are at least $\la (A-2)/100$ leaves in the $i$-th star of $\Gamma$ that are infected at time $ta$. We will show that
		\begin{equation}\label{eq:dominate:percolation}
			\P\left (\A_{i}^{t+1} \cap \A_{i-1}^{t+1}  \cap \A_{i+1}^{t+1} \big\vert \A_{i}^{t}  \right ) \ge 0.99 =:1 -16\ep.
		\end{equation}
		Consider the following process on $\L$ where we expose the edges as $t$ grows. At any time $ta$, we relabel the leaves so that the infected leaves in each star are placed first. We then say that the edges $(i, t)\to (j, t+1)$ with $j\in \{i\pm 1, i\}\cap [L]$ are open if $\A_{i}^{t+1} \cap \A_{i-1}^{t+1}  \cap \A_{i+1}^{t+1}$ assuming that the first $\la(A-2)/100$ leaves in the $i$-th star are infected at time $ta$. This is a random variable depending only on the Poisson clocks within the $i$-th and $(i\pm 1)$-th stars, so they are $3$-dependent. Note that this process satisfies the property that if a vertex $(i, t)$ is connected with some vertex $(j, 0)$ by a directed open path then there are at least $\la (A-2)/100$ leaves in the $i$-th star of $\Gamma$ that are infected at time $ta$. 
		By Theorem 0.0 (i) in \cite{liggett1997domination}, this process dominates the described percolation. Therefore, we obtain the stated lemma.

	To see \eqref{eq:dominate:percolation},	let $\mathcal B_i^t$ be the event that there are at least $\la (A-2)/100$ vertices in the $i$-th star of $\Gamma$ infected at time $ta+n$ for all positive integers $n\le a$. Applying Lemma \ref{lm:star:iteration} $a$ times, we get
		\begin{equation}
			\P\left (\A_{i}^{t+1}\big\vert \A_{i}^{t}  \right ) \ge \P\left ( \mathcal B_i^{t}\big\vert \A_{i}^{t}  \right )\ge 1- 3a\exp(-\la^{2}(A-2)/12800)\ge 1 - \ep.
		\end{equation}

		Note that each vertex infected at certain time $s$ has a probability of $\P(\Pois(1/2) = 0) \ge 1/e$ to remain infected until time $s+1/2$. Thus, under $\mathcal B_i^{t}$, by Chernoff inequality, for all positive integers $n\le a$, at least $M = 0.01 \la (A-2)/4$ vertices infected at time $ta+n$ remain infected until time $ta+n+1/2$ with probability at least $1 - a\exp(-\la (A-2)/10^4) \ge 1 - \ep$.
		Under this event, applying Lemma \ref{lm:star:rho} $a$ times for each of the time intervals $[ta+n, ta+n+1/2]$, the root $v_i$ of the $i$-th star is infected for more than $a/4$ time in the time interval $[ta, ta+a]$ with probability at least
		$$1- 2a\exp(-\la M/32) = 1- 2a\exp(-\la^{2} (A-2)/12800)\ge 1 -\ep.$$

		Thus, the probability that $v_i$ infects $v_{i+1}$ during $[ta, ta+a] $ is at least
		$$\P(\Pois(\la a/4)\neq 0) =1 - \exp(-\la a/4)\ge 1 - \ep,$$
		since $a \ge 40/\la$.
		
		Once $v_{i+1}$ is infected, in proving Lemma \ref{lm:star}, we proved that $\A_{i+1}^{t+1}$ occurs with probability at least
		$$1 - \frac{200 \log (\la (A-2))}{\la^{2}(A-2)} - 3e^{-\la^{2}(A-2) /10^5}\ge 1 - \ep.$$
		
		Combining all of these events by the union bound, we obtain
		\begin{equation}
		\P\left (\A_{i+1}^{t+1}\big\vert \A_{i}^{t}  \right ) \ge 1 - 5\ep.  \nonumber
		\end{equation}
		Similarly for $\A_{i-1}^{t+1}$. By the union bound, we obtain \eqref{eq:dominate:percolation}.
	\end{proof}

	\begin{lemma}[Long survival of percolation]\label{lm:percolation:survival}  With high probability, the percolation defined in Lemma \ref{lm:dominate:percolation} reaches time $1.01^{L}$. In other words, there exists a path of open edges connecting some vertex $(i, 0)$ and some vertex $(j, 1.01^{L})$ for some $i, j\in [L]$.
	\end{lemma}
	
	\begin{proof} [Proof of Lemma \ref{lm:percolation:survival}] We use the standard Peierls argument (\cite{peierls1936ising}). We say that a vertex $(i, t)\in [L]\times \N$ is occupied if there exists a directed open path connecting some vertex $(j, 0)$ and $(i, t)$. Observe that the set $\mathcal C$ of occupied vertices is connected, as a subgraph of the lattice $\Z^{2}$ (rather than $\L$). Consider the subset $\mathcal C' = (\mathcal C + [-1/2, 1/2]^{2})\cap [0, L]\times [0, \infty)$ which is formed by surrounding each vertex in $\mathcal C$ by the box $[-1/2, 1/2]^{2}$. Let $B$ be the upper boundary of $\mathcal C'$ inside $[0, L]\times [0, \infty)$. Equivalently, $B$ is the path that connects the top left-most point and the top right-most point of $\mathcal C'$; note that in forming $B$, we do not at all look at the edges of $\L$ but rather think about $\mathcal C$ as a subset of the regular lattice $\Z^{2}$. See Figure \ref{fig:B}.
		\begin{figure}[H]
			\begin{center}
				\begin{tikzpicture}[scale=0.5]
				\foreach \x in {0,1,2,3,4,5,6,7}
								\foreach \y in {0,1,2,3, 4}
				\node[circle, fill, black, draw, inner sep=1pt]  at (2*\x,2*\y) {};
				
				\foreach \x in {0,1,2,3,4,5,6,7}
 			\node[circle, fill, green, draw, inner sep=1pt]  at (2*\x,0) {};
				
				\node[circle, fill, green, draw, inner sep=1pt]  at (0,2) {};
				\node[circle, fill, green, draw, inner sep=1pt]  at (0,4) {};

				\node[circle, fill, green, draw, inner sep=1pt]  at (4,2) {};
				\node[circle, fill, green, draw, inner sep=1pt]  at (4,4) {};
				\node[circle, fill, green, draw, inner sep=1pt]  at (4,6) {};
				
				\node[circle, fill, green, draw, inner sep=1pt]  at (6,6) {};

						\node[circle, fill, green, draw, inner sep=1pt]  at (8,4) {};
				\node[circle, fill, green, draw, inner sep=1pt]  at (8,6) {};

				\node[circle, fill, green, draw, inner sep=1pt]  at (12,2) {};
				\node[circle, fill, green, draw, inner sep=1pt]  at (12,4) {};
				\node[circle, fill, green, draw, inner sep=1pt]  at (14,2) {};
				\node[circle, fill, green, draw, inner sep=1pt]  at (14,4) {};
				
				\node[circle, fill, green, draw, inner sep=1pt]  at (14,6) {};
			 
				\draw[red, postaction={decorate,decoration={markings, mark=at position 0.5 with {\arrow{>}}}}] (0, 5) -- (1, 5);
				\node[circle, fill, red, draw, inner sep=1pt]  at (0,5) {};
				\node[circle, fill, red, draw, inner sep=1pt]  at (1,5) {};
				\draw[red, postaction={decorate,decoration={markings, mark=at position 0.5 with {\arrow{>}}}}] (1, 5) -- (1, 3);
				\node[circle, fill, red, draw, inner sep=1pt]  at (1,3) {};
				\draw[red, postaction={decorate,decoration={markings, mark=at position 0.5 with {\arrow{>}}}}] (1, 3) -- (1, 1);
				\node[circle, fill, red, draw, inner sep=1pt]  at (1,1) {};
				\draw[red, postaction={decorate,decoration={markings, mark=at position 0.5 with {\arrow{>}}}}] (1, 1) -- (3, 1);
				\node[circle, fill, red, draw, inner sep=1pt]  at (3,1) {};

				\draw[red, postaction={decorate,decoration={markings, mark=at position 0.5 with {\arrow{>}}}}] (3, 1) -- (3, 3);
				\node[circle, fill, red, draw, inner sep=1pt]  at (3,3) {};
				\draw[red, postaction={decorate,decoration={markings, mark=at position 0.5 with {\arrow{>}}}}] (3, 3) -- (3, 5);
				\node[circle, fill, red, draw, inner sep=1pt]  at (3,5) {};
				\draw[red, postaction={decorate,decoration={markings, mark=at position 0.5 with {\arrow{>}}}}] (3, 5) -- (3, 7);
				\node[circle, fill, red, draw, inner sep=1pt]  at (3,7) {};

				\draw[red, postaction={decorate,decoration={markings, mark=at position 0.5 with {\arrow{>}}}}] (3, 7) -- (5, 7);
				\node[circle, fill, red, draw, inner sep=1pt]  at (5, 7) {};
				\draw[red, postaction={decorate,decoration={markings, mark=at position 0.5 with {\arrow{>}}}}] (5, 7) -- (7, 7);
				\node[circle, fill, red, draw, inner sep=1pt]  at (7, 7) {};
				\draw[red, postaction={decorate,decoration={markings, mark=at position 0.5 with {\arrow{>}}}}] (7, 7) -- (9, 7);
				\node[circle, fill, red, draw, inner sep=1pt]  at (9,7) {};

				\draw[red, postaction={decorate,decoration={markings, mark=at position 0.5 with {\arrow{>}}}}] (9, 7) -- (9, 5);
				\node[circle, fill, red, draw, inner sep=1pt]  at (9, 5) {};
				\draw[red, postaction={decorate,decoration={markings, mark=at position 0.5 with {\arrow{>}}}}] (9, 5) -- (9, 3);
				\node[circle, fill, red, draw, inner sep=1pt]  at (9,3) {};
				
				\draw[red, postaction={decorate,decoration={markings, mark=at position 0.5 with {\arrow{>}}}}] (9, 3) -- (7, 3);
				\node[circle, fill, red, draw, inner sep=1pt]  at (7,3) {};

				\draw[red, postaction={decorate,decoration={markings, mark=at position 0.5 with {\arrow{>}}}}] (7, 3) -- (7, 5);
				\node[circle, fill, red, draw, inner sep=1pt]  at (7,5) {};

				\draw[red, postaction={decorate,decoration={markings, mark=at position 0.5 with {\arrow{>}}}}] (7, 5) -- (5, 5);
				\node[circle, fill, red, draw, inner sep=1pt]  at (5,5) {};

				\draw[red, postaction={decorate,decoration={markings, mark=at position 0.5 with {\arrow{>}}}}] (5, 5) -- (5, 3);
				\node[circle, fill, red, draw, inner sep=1pt]  at (5,3) {};	
				\draw[red, postaction={decorate,decoration={markings, mark=at position 0.5 with {\arrow{>}}}}] (5, 3) -- (5, 1);
				\node[circle, fill, red, draw, inner sep=1pt]  at (5,1) {};

				\draw[red, postaction={decorate,decoration={markings, mark=at position 0.5 with {\arrow{>}}}}] (5, 1) -- (7, 1);
				\node[circle, fill, red, draw, inner sep=1pt]  at (7, 1) {};	
				\draw[red, postaction={decorate,decoration={markings, mark=at position 0.5 with {\arrow{>}}}}] (7, 1) -- (9, 1);
				\node[circle, fill, red, draw, inner sep=1pt]  at (9, 1) {};
				\draw[red, postaction={decorate,decoration={markings, mark=at position 0.5 with {\arrow{>}}}}] (9, 1) -- (11, 1);
				\node[circle, fill, red, draw, inner sep=1pt]  at (11, 1) {};

				\draw[red, postaction={decorate,decoration={markings, mark=at position 0.5 with {\arrow{>}}}}] (11, 1) -- (11, 3);
				\node[circle, fill, red, draw, inner sep=1pt]  at (11, 3) {};
				\draw[red, postaction={decorate,decoration={markings, mark=at position 0.5 with {\arrow{>}}}}] (11, 3) -- (11, 5);
				\node[circle, fill, red, draw, inner sep=1pt]  at (11, 5) {};

				\draw[red, postaction={decorate,decoration={markings, mark=at position 0.5 with {\arrow{>}}}}] (11, 5) -- (13, 5);
				\node[circle, fill, red, draw, inner sep=1pt]  at (13, 5) {};
				
				\draw[red, postaction={decorate,decoration={markings, mark=at position 0.5 with {\arrow{>}}}}] (13, 5) -- (13, 7);
				\node[circle, fill, red, draw, inner sep=1pt]  at (13, 7) {};
				
				\draw[red, postaction={decorate,decoration={markings, mark=at position 0.5 with {\arrow{>}}}}] (13, 7) -- (14, 7);
				\node[circle, fill, red, draw, inner sep=1pt]  at (14, 7) {};

				\end{tikzpicture}
			\end{center}
			\caption{An example of $\mathcal C$ which consists of green points and its upper boundary $B$ which is the red path. The black points are points in $\L$ that do not belong to $\mathcal C$. }\label{fig:B}
		\end{figure}
		We observe that if the percolation does not reach time $1.01^L$ then $B$ lies entirely in $[L]\times [0, 1.01^L]$.
		
		Let $\alpha$ be the length of $B$. To bound the probability of having a path $B$ of length $\alpha$, we want to find a lower bound for the number of distinct closed edges in $\L$ that corresponds to $B$. Walking from the left-most vertex on $B$ gives an orientation to each of the edges in $B$. Let $\alpha_1, \alpha_2, \alpha_3, \alpha_4$ be the number of edges in $B$ that are in the rightward, leftward, upward, and downward directions, respectively. We have $\alpha = \sum_{i=1}^{4}\alpha_i.$
		
		A right-ward edge in $B$ must occur in this situation \begin{tikzpicture}[baseline=(current bounding box.center)]
		\node[circle, fill, green, draw, inner sep=1pt]  at (0,-0.5) {};
		\node[circle, fill, black, draw, inner sep=1pt]  at (0,0.5) {};
		\draw[red, postaction={decorate,decoration={markings, mark=at position 0.5 with {\arrow{>}}}}] (-0.5, 0) -- (0.5, 0);
		\node[circle, fill, red, draw, inner sep=1pt]  at (0.5, 0) {};
		\node[circle, fill, red, draw, inner sep=1pt]  at (-0.5, 0) {};
		\end{tikzpicture} where the green node right below is occupied and the black node right above is not. This means that the vertical edge from the green node to the black node in $\L$ is closed. Therefore, $B$ determines uniquely $\alpha_1$ vertical closed edges in $\L$.
		
		We further decompose $\alpha_3$ as follows. Let $\alpha_{3, 1}, \alpha_{3, 2}, \alpha_{3, 3}$ be the number of upward edges that follow an edge that is rightward, leftward,  upward, respectively. We have
		$$\alpha_3 = \alpha_{3, 1}+ \alpha_{3, 2}+ \alpha_{3, 3}.$$
		Similarly,  let $\alpha_{4, 1}, \alpha_{4, 2}, \alpha_{4, 4}$ be the number of downward edges that follow a rightward, leftward, downward edge, respectively. We have
		$$\alpha_4 = \alpha_{4, 1}+ \alpha_{4, 2}+ \alpha_{4, 4}.$$

		An edge $e$ counted in $\alpha_{3, 3}$ must occur in either this situation \begin{tikzpicture}[baseline=(current bounding box.center)]
		\node[circle, fill, green, draw, inner sep=1pt, label=right:$u$]  at (0.25, -0.25) {};
		\node[circle, fill, green, draw, inner sep=1pt]  at (0.25, 0.25) {};
		\node[circle, fill, black, draw, inner sep=1pt, label=left:$v$]  at (-0.25, 0.25) {};
		\node[circle, fill, black, draw, inner sep=1pt]  at (-0.25, -0.25) {};
		\draw[red, postaction={decorate,decoration={markings, mark=at position 0.5 with {\arrow{>>}}}}] (0, 0) -- (0, 0.5);
		\draw[red, postaction={decorate,decoration={markings, mark=at position 0.5 with {\arrow{>}}}}] (0, -0.5) -- (0, 0);
		\node[circle, fill, red, draw, inner sep=1pt]  at (0, 0.5) {};
		\node[circle, fill, red, draw, inner sep=1pt]  at (0, -0.5) {};
		\node[circle, fill, red, draw, inner sep=1pt]  at (0, 0) {};
		\end{tikzpicture} or this situation \begin{tikzpicture}[baseline=(current bounding box.center)]
		\node[circle, fill, black, draw, inner sep=1pt]  at (0.25, -0.25) {};
		\node[circle, fill, black, draw, inner sep=1pt, label=right:$v$]  at (0.25, 0.25) {};
		\node[circle, fill, green, draw, inner sep=1pt]  at (-0.25, 0.25) {};
		\node[circle, fill, green, draw, inner sep=1pt, label=left:$u$]  at (-0.25, -0.25) {};
		\draw[red, postaction={decorate,decoration={markings, mark=at position 0.5 with {\arrow{>>}}}}] (0, 0) -- (0, 0.5);
		\draw[red, postaction={decorate,decoration={markings, mark=at position 0.5 with {\arrow{>}}}}] (0, -0.5) -- (0, 0);
		\node[circle, fill, red, draw, inner sep=1pt]  at (0, 0.5) {};
		\node[circle, fill, red, draw, inner sep=1pt]  at (0, -0.5) {};
		\node[circle, fill, red, draw, inner sep=1pt]  at (0, 0) {};
		\end{tikzpicture}.
	Here $e$ is the edge with double arrows. Note that the occupied (green) vertices lie in the finite connected component of $\Z^{2}$ separated by $B$; so, $B$ determines which of the situations occur together with the lower occupied vertex $u$ and the upper unoccupied vertex $v$. The edge from $u$ to $v$ must be closed in $\L$. Thus, $B$ determines uniquely $\alpha_{3, 3}$ closed diagonal edges in $\L$. Similarly, there are $\alpha_{4, 4}$ closed diagonal edges in $\L$ determined by $B$ which are distinct from the aforementioned $\alpha_{3, 3}$ edges.

	Observe furthermore that
		\begin{enumerate}
			\item $\alpha_2\le \alpha_1$ because the directed path $B$ goes from its left-most vertex to its right-most one;
			\item $\alpha_{3, 1}+\alpha_{4, 1}\le \alpha_1$
			\item $\alpha_{3, 2}+\alpha_{4, 2}\le \alpha_2\le \alpha_1$.
		\end{enumerate}
		Thus,
		\begin{eqnarray}
			\alpha &=&(\alpha_1+ \alpha_{3, 3} + \alpha_{4, 4})+\alpha_2+ \alpha_{3, 1}+ \alpha_{3, 2}+\alpha_{4, 1}+ \alpha_{4, 2}\nonumber\\
			&\le& (\alpha_1+ \alpha_{3, 3} + \alpha_{4, 4}) + 3\alpha_1 \le 4 (\alpha_1+ \alpha_{3, 3} + \alpha_{4, 4}).\nonumber
		\end{eqnarray}
		Therefore, the probability that a given path $B$ occurs is at most $0.01^{\alpha/4}$.
		Given $\alpha$, the number of paths $B$ of length $\alpha$ is at most $1.01^L3^{\alpha}$, where the first factor is the number of ways to choose the left-most vertex and at any point on the path, there are at most 3 ways to determine the next direction.
		By the union bound, the probability that the percolation does not reach time $1.01^L$ is at most
		$$\sum_{\alpha = L}^{\infty} 1.01^L \cdot 3^{\alpha} 0.01^{\alpha/4} \ll 1.01^L \cdot 0.95^L = o(1),$$
		completing the proof.
\end{proof}

  \section{Survival time on trees}\label{sec:upper:tree}
  To prove Theorem~\ref{thm:uper}, we first derive an upper bound on the survival time on trees. Following the methods in \cite{BNNS} and \cite{NNS}, we can only get an upper bound of $n^{\la^{1+o(1)}}$ rather than $n^{\la^{2+o(1)}}$. The key new idea here is that as $\la$ is small compared to the degree of a typical vertex, a significant portion of infections are {\it weak} in the sense the vertex infected recovers before spreading the infection. Therefore, we introduce a modified process that focuses on {\it strong} infections and this allows us to get the power $\la^{2}$ in the upper bound.

  \subsection{The coupling with the starred contact process}
  In this section, we introduce the {\it starred} contact process which takes into account whether a new infection at a vertex $u$ is {\it strong}. If we ignore the weak infections, the contact process basically remains the same except that at any time, there could be vertices that are infected and then recover at a later time without passing the infection to any other vertices.

  To be specific, let us sample the Poisson clocks for recoveries and infections on $G_n$. Using these clocks, the contact process $(X_t)_{t\ge 0}$ with $X_0=S\subset V$ is read from the graphical representation (See Section \ref{sec:graphical}). We now construct the following Markov process $(X^*_t)_{t\ge 0}$ which we call the starred contact process from the same set of clocks.

  At time $0$, $X^*_0(u)=0$ if one of the following hold
  \begin{itemize}
  	\item $u\notin S$,
  	\item $u\in S$ and the recovery clock at $u$ rings before any infection clocks from $u$ to its neighbors. In other words, $u$ recovers before sending any infection out.
  \end{itemize}
Otherwise, we say that $u$ is strongly infected and set $X^*_0(u)=1^*$. All vertices remain in their states until any of the clocks ring, say at time $t$. Consider the following cases depending on the nature of the clock.

  \begin{itemize}
  	\item  If the clock is a recovery at a vertex $u$ then set $X^*_t(u)=0$.
  	\item If the clock is an infection from a healthy vertex $v$ (that is $X^*_{t^-}(v)=0$) to another vertex $u$, nothing changes.
  	\item If the clock is an infection from an infected vertex $v$ (that is $X^*_{t^-}(v)$ equals 1 or $1^*$) to another infected vertex $u$, we set $X^*_t(v)=1$.
  	\item If the clock is an infection from an infected vertex $v$ (that is $X^*_{t^-}(v)$ equals 1 or $1^*$) to a healthy vertex $u$, then if the first clock after time $t$ at $u$ is a recovery clock rather than an infection clock from $u$ to any of its neighbors, then $u$ remains healthy at time $t$. Otherwise, set $X^*_t(u)=1^*$. In either cases, set $X^*_t(v)=1$.
  \end{itemize}
In any cases, the state at other vertices remain the same.
  We observe that for any time $t$ and any vertex $u$, if $X_t(u)=1$ and $X^*_t(u)=0$ then in the process $(X_s)$, after this time, $u$ will recover before infecting any other vertices. Denote by $\mathbf R$ and $\mathbf R^*$ the survival time of the processes $(X_t)$ and $(X^*_t)$, respectively. The above observation implies that
  $$\mathbf R\le \mathbf R^*+\max_{i\leq n} \tau_i,$$
  where $\tau_1, \dots, \tau_n$ are iid $\Exp(1)$ random variables. The sum $\sum_{i=1}^{n} \tau_i$ bounds the total healing time of all vertices $u$ with $X_{\mathbf R^*}(u)=1$.

  Since \whp, $\max_{i\leq n} \tau_i = O(\log n)$, we conclude that for all $G_n$, \whp,
  \begin{equation}\label{eq:unstar}
  \mathbf R\le \mathbf R^*+O(\log n).
  \end{equation}

  Thus, it suffices to bound $\mathbf R^*$. For the rest of this paper, we shall only work on the starred contact process and so, for simplicity, we write $(X_t)$ for the starred contact process from now on.

  \subsection{Recursive survival time on trees}
  Before analyzing the process on graphs, we study it in the simpler case of trees. Then in passing from trees to graphs we will utilizes the fact that Galton-Watson trees are local weak limits of random graphs as discussed in Section~\ref{subsubsec:rgprelim}.

  Let $T$ be a (deterministic) tree and consider the starred process on $T$ with $\rho$ initially at $1^*$. Let $\mathbf R(T)$ be the survival time of this process. Let $R(T)=\E_{\CP} \mathbf R(T)$.

  Let $T^+$ be the tree $T$ together with an added vertex $\rho^+$ which is the parent of $\rho$. Consider the starred contact process on a tree $T^+$ with $\rho^+$ always infected and $\rho$ initially at state $1^*$ which means that $\rho$ will first send an infection to either $\rho^+$ or its children. Let $\bS(T)$ be the first time that the process returns to the all healthy state on $T$ and $S(T) = \E _{\CP} \bS(T)$ be the expected excursion time. The addition of $\rho^+$ makes sure that the contact process on $T^+$ never dies out and so the process no longer has an absorbing state.

  Note that if $\rho$ has at least 1 child, then with probability at least $\frac12$, the first infection from $\rho$ is sent to its children rather than to $\rho^+$. Thus,
  $$R(T)\le 2 S(T).$$

  Let $D$ be the number of children of $\rho$ and $\rho_1, \dots, \rho_D$ be these children. Let $d_i$ be the number of children of $\rho_i$. Let $T_i$ be the subtree of $T$ containing $\rho_i$ and all of its successors.

  An infection from $\rho$ to $\rho_i$ is strong if $\rho_i$ sends an infection to either $\rho$ or one of its children before recovering. This happens with probability $\frac{\la (d_i+1)}{1+\la (d_i+1)}$. We call this number the weight of $T_i$, denoted by $\wt(T_i)$,
  \begin{equation}\label{def:w:i}
  \wt(T_i)=\frac{\la (d_i+1)}{1+\la (d_i+1)}.
  \end{equation}
Similarly,
  \begin{equation}\label{def:w:G}
	\wt(T) = \frac{\la (D+1)}{1+\la (D+1)}.
\end{equation}

  The following recursion allows us to control the growth of $S(T)$.
  \begin{lemma}[Recursion for $S$]\label{lm:S:recursion} We have
  	\begin{eqnarray}
  	S(T) &\le& \frac{1}{1+\la (D+1)} - \frac{1}{\la(D+1)}+  \wt(T)^{-1} \prod_{i=1}^{D} (1 + \la \wt(T_i) S(T_i)).\nonumber
  	\end{eqnarray}
  \end{lemma}

  \begin{proof}
  	Let $(X_t)$ be the starred contact process on $T^+$ with $\rho^+$ always infected and $\rho$ initially at $1^*$. We consider the so-called root-added process $(X^\sharp_t)$ on $T^+$ that also has $\rho^+$ always infected, also starts at $\rho$ being $1^*$ and generated by the same recovery and infection clocks as $(X_t)$ except that we ignore any recoveries of $(X_t)$ at $\rho$ if at least one of its successors is still infected at the time.
  	Let $ S^\sharp(T)$ be the expected excursion time of $(X^\sharp_t)$ that is the expected time difference from the time $\rho$ is strongly infected to the first time that all vertices in $T$ recover.
  	Since $S(T)\le S^{\sharp} (T)$ by Lemma \ref{lem:graph rep2}, it suffices to show that
  	\begin{eqnarray} \label{eq:lm:S:sharp}
  	S^{\sharp}(T) &=& \frac{1}{1+\la (D+1)} - \frac{1}{\la(D+1)}+  \wt(T)^{-1} \prod_{i=1}^{D} (1 + \la \wt(T_i) S(T_i)).
  	\end{eqnarray}
  	Let $t_0$ be the first time that an event at $\rho$ happens, which could either be a recovery at $\rho$ or an infection from $\rho$. We have
  	\begin{equation}
  	\E t_0 = \frac{1}{1+\lambda (D+1)}.\nonumber
  	\end{equation}
  	Since $\rho$ is at $1^*$ initially, this first event has to be an infection from $\rho$ to a child $\rho_i$ or to $\rho^+$. Note that conditioning on the event that $\rho$ is $1^*$ at time $0$ does not change the expectation of $t_0$.
  	
  	If $\rho $ infects $ \rho_i$ at time $t_0$, then the healing clock at $\rho$ is deactivated and the infection is spread across $T$ until all of $T\setminus \{\rho\}$ recovers. Denote by $t_1$ the first time that this happens. If at time $t_0$, $\rho $ infects $ \rho^+$ then $t_1 =t_0$. 	
  	
%
%
Consider the starred contact process on $T$ with $\rho$ always infected and one of its children initially at $1^*$. More precisely, let $\cp_{\rho} (T_i;\one^*_{\rho_i})$ (and $\cp_{\rho} (T_i;\zero)$) (the subscript $\rho$ indicates that it serves as the added parent above $\rho_i$) be the root-added contact processes on each $T_i$ with $\rho_i$ initially at $1^*$ (and with all vertices in $T_i$ initially healthy, respectively). Consider their product
\begin{equation}\label{eq:def:productchain}
\cp^{\otimes}(T;\one^*_{\rho_i}) :=
\left(\otimes_{\substack{j=1\\j\neq i}}^D \cp_{\rho} (T_j;\zero ) \right)
\otimes
\cp_{\rho} (T_i;\one^*_{\rho_i}),
\end{equation}
for each $i\in[D]$. Let $\textbf{S}_i^\otimes$ denote the  excursion time of this process, that is, the first return time to the all-healthy state $\otimes_{j=1}^D \zero_{T_j}$, and let $S_i^\otimes = \E \textbf{S}_i^\otimes$. Further, define the average of $S_i^\otimes$ by
\begin{equation} 
S^\otimes = \frac{1}{D} \sum_{i=1}^D S_i^\otimes.\notag
\end{equation}
  	Let
  	\begin{eqnarray}\label{eq:theta}
  	\theta &=& \frac{1}{D} \sum_{i=1}^{D} \P(\text{the infection from $\rho$ to $\rho_i$ is a strong infection}\vert \rho \text{ infects }\rho_i \text{ at time $t_0$})\notag\\
  	 &=&  \frac{1}{D} \sum_{i=1}^{D}  \wt(T_i)
  	\end{eqnarray}
  	be the probability that an infection from $\rho$ to its children is strong.
  	We observe that
  	\begin{equation}\label{eq:t1:t0}
  	\E (t_1 - t_0\vert \rho \text{ sends an infection to a child at time } t_0) = \theta S^\otimes .
  	\end{equation}
  	
%
%
%
  	
  	Using \eqref{eq:t1:t0}, we shall now prove that
  	\begin{equation}\label{eq:S:T}
  	S^{\sharp} (T) =  \E t_0 + \E (t_1- t_0) + \sum _{k=0}^{\infty} \left (\frac{\la D}{1+\la D}\right )^{k} \left (\frac{1}{1+\la D}\right ) \left (\frac{k+1}{1+\la D}+ k \theta S^\otimes  \right ).
  	\end{equation}

  	To see \eqref{eq:S:T}, at time $t_1$, only $\rho$ and $\rho^+$ remain infected. The next meaningful event would either be a recovery at $\rho$ or an infection from $\rho$ to a child $\rho_i$. The former corresponds to $k=0$. If the latter happens instead, say at a time $t_2$, then the infection will spread across $T$ again until everything in $T\setminus \{\rho\}$ is healed. Now, another attempt to heal $\rho$ happens again with probability of success $\frac{1}{1+\la D}$. If success, we set $k=1$. Otherwise, the whole process repeats and $k$ is correspondingly the number of fail attempts to heal $\rho$. The probability of $k$ is $\left (\frac{\la D}{1+\la D}\right )^{k} \left (\frac{1}{1+\la D}\right )$ and the time spent is $\frac{k+1}{1+\la D}+ k \theta S^\otimes$ according to \eqref{eq:t1:t0}.
  	
  	  	By algebraic manipulation, we obtain
  	\begin{eqnarray}\label{eq:S:T:2}
  	S^{\sharp} (T) &=& \frac{1}{1+\la (D+1)}+ \frac{D}{D+1}\theta S^\otimes + \sum _{k=0}^{\infty} \left (\frac{\la D}{1+\la D}\right )^{k} \left (\frac{1}{1+\la D}\right ) \left (\frac{k+1}{1+\la D}+ k \theta S^\otimes \right )\nonumber\\
  	&=& \frac{1}{1+\la (D+1)}+ \frac{D}{D+1}\theta S^\otimes + 1 + \la D  \theta S^\otimes\nonumber\\
  	&=&\frac{1}{1+\la (D+1)} - \frac{1}{\la(D+1)}+  \frac{1+\la (D+1)}{\la (D+1)} (1+\la D\theta S^\otimes).
  	\end{eqnarray}
  	
  	We now evaluate $S^{\otimes}$ using the observation that $(X^{\otimes}_t)$, as $\rho$ is always infected, is just the product of the root-added contact processes on the $T_i$'s, where $\rho$ plays the role of $\rho_i^{+}$ -- the added-root. Let $\pi_i$ be the stationary distribution of the root-added starred contact process on $T_i$. Then, $\pi_i({\bf 0})$ is the fraction of time that $T_i$ is all healthy and so
  	$$\pi_i({\bf 0}) = \frac{\left (\la \frac{\la (d_i+1)}{1+\la (d_i+1)}\right )^{-1}}{\left (\la \frac{\la (d_i+1)}{1+\la (d_i+1)}\right )^{-1} + S(T_i)} = \frac{1}{1 + \la \wt(T_i) S(T_i)},$$
  	where $\la \frac{\la (d_i+1)}{1+\la (d_i+1)}$ is the rate at which $\rho$ sends a strong infection to $\rho_i$.
  	
  	Similarly, let $\pi$ be the stationary distribution of $(X^{\otimes}_t)$. Then, $\pi({\bf 0})$, the fraction of time that $T\setminus \{\rho\}$ is all healthy, is
  	$$\pi({\bf 0}) = \frac{\left (\la D\theta\right )^{-1}}{\left (\la  D\theta \right )^{-1} + S^\otimes} = \frac{1}{1 + \la  D\theta S^\otimes},$$
  	where $\la D \theta$ is the rate at which $\rho$ sends a strong infection to one of the $\rho_i$.

  	Since $\pi = \prod_{i=1}^{D} \pi_i$, we have
  	$$\frac{1}{1 + \la D \theta S^\otimes} = \prod_{i=1}^{D} \frac{1}{1 + \la \wt(T_i) S(T_i)}$$
  	and so,
  	\begin{equation}\label{eq:S:prod}
  	1 + \la D \theta S^\otimes = \prod_{i=1}^{D} (1 + \la \wt(T_i) S(T_i)).
  	\end{equation}
  	Combining this with \eqref{eq:S:T:2}, we get \eqref{eq:lm:S:sharp}.
  \end{proof}

  \subsection{Recursion for the total number of infections at the leaves}
   To pass from Galton-Watson trees to random graphs, it is crucial to control the effect of cycles as they allow the infection to transfer from one local neighborhood to another. To this end, one needs to control the number of times the infection hits the boundary of a local neighborhood. Therefore, following \cite{NNS}, in this section, we study the \textit{total  infections at leaves} of (finite) trees defined as follows.

  \begin{definition}[Total infections at leaves]\label{def:tit}
  	Let $T$ be a finite tree rooted at $\rho$, set
  	 \begin{equation}\label{eq:depth def}
  	l=\max\{\textnormal{dist}(\rho,v):v\in T \}.
  	\end{equation}
  	 be the depth of the tree and $$\mathcal{L}:=\{v\in T: \textnormal{dist}(\rho,v)=l \}  $$
  	be the collection of depth-$l$ leaves of $T$. Suppose that $l\geq 1$ and consider the root-added, starred contact process $(X_t)$ on $T^+ = T\cup \{\rho^+\}$ with $\rho^+$ always infected and $\rho$ initially at state $1^*$. For $v\in \mathcal{L}$, define the \textit{total infections at} $v$, by
  	\begin{equation*}
  	\textbf{M}_{l, v}(T):=\,\textnormal{the number of infections at } v \textnormal{ in }(X_t) \textnormal{ during time } t\in[0,\textbf{S}(T)],
  	\end{equation*}
  	where  $\textbf{S}(T)$ is the first time that $(X_t)$ returns to the all healthy state on $T$. In other words, we count the number of times $t$ such that $X_t(v)=1^*$ and $X_{t-}(v)=0$ for $t\leq \textbf{S}(T)$. Then,
  	we define  the \textit{total infections at depth-$l$ leaves} (\textit{during a single excursion}) by
  	\begin{equation*}
  	\textbf{M}_{l}(T) = \sum_{v \in \mathcal{L}} \textbf{M}_{l, v}(T).
  	\end{equation*}
  	For $l'>l$, we set $\textbf{M}_{l'}(T)\equiv 0$.
  	
  	We also denote the \textit{expected total  infections at depth-$l$ leaves} by $M_{l}(T) = \E_{\textsc{cp}} \textbf{M}_{l}(T)$. Also, as above, we write $M_{l'}(T)=0$ for $l'>l$. Moreover, if the tree depth is $0$  (that is, $T$ consists of a single vertex), we set $\textbf{M}_{0}(T) \equiv 1$.
  \end{definition}

  Let $(T,\rho)$ be a finite rooted tree of \textit{depth} $l$.
  The goal of this section is to establish the following recursive bound.
    \begin{lemma}\label{lm:Mrecursion}
  	For a finite rooted tree $(T,\rho)$ of depth $l$, let $D=\deg(\rho)$ and $T_1,\ldots,T_D$ be the subtrees rooted at the children $v_1, \dots, v_D$ of $\rho$. Then, $M_{l}(T)$, the expected total infections at depth-$l$ leaves  on $T$, satisfies 
  	\begin{equation}\label{eq:Mrecursion atypical}
  	\wt(T) M_{l}(T) \leq  \la \sum_{i=1}^D \wt(T_i) M_{l-1}(T_i) \prod_{\substack{1\,\leq\, j\, \leq\, D\\ j\neq i}}   \left (1+\lambda \wt(T_i) S(T_j)\right ),
  	\end{equation}
  	where we recall that
  	\begin{equation}\label{eq:wi}
  	\wt(T_i) = \frac{\la(d_i+1)}{1+\la (d_i+1)}\quad\text{and}\quad  \wt(T) = \frac{\la(D+1)}{1+\la (D+1)}.
  	\end{equation}
  \end{lemma}

  \begin{proof} [Proof of Lemma \ref{lm:Mrecursion}]
  	We shall use the same notations as in the proof of Lemma \ref{lm:S:recursion}. In particular, we consider the following processes on $T$:
  	\begin{itemize}
  		\item $(X_t)$ the root-added process with $\rho$ initially at state $1^*$,
  		\item $(X^{\sharp}_t)$ the root-added process with $\rho$ initially at state $1^*$ and the recovery clock at $\rho$ being temporary disabled if $T\setminus \{\rho\}$ is not all healthy,
  		\item $(X^{\otimes}_{i, t})$ the process with $\rho$ always infected and the $\rho_i$ initially at state $1^*$. That means $(X^{\otimes}_{i, t})\sim \cp^{\otimes}(T;\one^*_{\rho_i})$ as in \eqref{eq:def:productchain}.
  	\end{itemize}
  	We denote the expected total infections at depth-$l$ leaves by $M_{l}(T)$, $M^\sharp(T)$ and $M_i^\otimes(T)$, correspondingly. We define 
  	$$M^\otimes(T) = \frac{1}{D}\sum_{i=1}^{D} M_i^\otimes(T).$$
  	By a standard coupling between $(X^\sharp_t)$ and $(X_t)$ based on their graphical representations, we have $M_{l}(T)\le M^\sharp(T) $.
  	
  	During one excursion of $(X^{\sharp}_t)$, $\rho$ first sends an infection to either $\rho^+$ or one of its children $\rho_i$. The probability that this is a strong infection to one of the $\rho_i$ is $\frac{D}{D+1}\theta$ where $\theta$ is defined in \eqref{eq:theta}. Between this first infection and the end of the excursion of $(X^{\sharp}_t)$, the number of excursions of $(X^{\otimes}_{i, t})$ is a geometric random variable with ``success" probability $\frac{1}{1+\la D \theta}$. This is because $\rho$ recovers at rate $1$ and sends a strong infection to the $\rho_i$ with rate $\la D\theta$; and the excursion of $(X^{\sharp}_t)$ ends precisely when $\rho$ recovers. So, the expected number of excursions of $(X^{\otimes}_{i, t})$ in one excursion of $(X^{\sharp}_t)$ is $\left ( \frac{D}{D+1} + \lambda D\right ) \theta$. Thus,
  	\begin{eqnarray}\label{eq:Msharp and Motimes}
  	M^\sharp(T)  &=& \frac{D\theta}{D+1}M^{\otimes}(T) +\sum_{k=0}^{\infty} \left (\frac{\lambda D\theta}{1+\lambda D\theta}\right )^{k} \frac{1}{1+\lambda D\theta}k M^{\otimes}(T)\notag\\
  	&=& \left ( \frac{1}{D+1} + \lambda \right ) D\theta M^\otimes(T).
  	\end{eqnarray}
  	
  	Now we control $M^\otimes(T)$ in terms of $\{M_{l-1}(T_i) \}$. Let $X^\otimes_t$ be the contact process on $T$ with $\rho$ always infected and all other vertices initially healthy. We observe that
  	\begin{equation}\label{eq:Motimes integral}
  	\begin{split}
  	&\lim_{t\rightarrow \infty} \frac{1}{t} \sum_{v\in\mathcal{L}}\E_{\textsc{cp}}
  	\left[\,
  	\left|\left\{ s \in [0,t]:\; X_{s}^\otimes(v)=1^* \textnormal{ and } X_{s-}^\otimes (v)=0 \right\} \right| \,
  	\right]\\
  	&=\lim_{t\rightarrow \infty} \frac{\text{the number of excursions of $(X^{\otimes}_s)$ in $[0, t]$}}{t} M^{\otimes}(T)  \\
  	&=
  	\frac{M^\otimes(T)}{(\lambda D\theta)^{-1} + S^\otimes(T)}.
  	\end{split}
  	\end{equation}
  	On the other hand, let $\mathcal{L}_i = \{v\in T_i: \textnormal{dist}(v,\rho_i) = l-1 \}$ for each $i\in [D]$ and $(X^{(i)}_t)$ be the restriction of $(X^\otimes_t)$ on $T_i$. We observe that  
  	\begin{equation*}
  	\begin{split}
  	&\sum_{v\in\mathcal{L}}\left|\left\{ s \in [0,t]:\; X_{s}^\otimes(v)=1^* \textnormal{ and } X_{s-}^\otimes (v)=0 \right\} \right| \\
  	&= \sum_{i=1}^D\sum_{v\in \mathcal{L}_i}
  	\left|\left\{ s \in [0,t]:\; X_{s}^{(i)}(v)=1^* \textnormal{ and } X_{s-}^{(i)} (v)=0 \right\} \right|.
  	\end{split}
  	\end{equation*}
  	Thus, the identity in (\ref{eq:Motimes integral}) is also equal to
  	\begin{equation*}
  	\sum_{i=1}^D \frac{M_{l-1}(T_i)}{(\lambda \wt(T_i))^{-1} + S(T_i)},
  	\end{equation*}
  	where $\wt(T_i)$, as defined in \eqref{def:w:i}, is the probability that an infection from $\rho$ to $\rho_i$ is strong.
  	Therefore,
  	\begin{eqnarray*}
  		M^\otimes(T) &=& \frac{1}{\lambda D\theta} \left (1+\lambda D\theta S^\otimes (T)\right )    \sum_{i=1}^D \la \wt(T_i) M_{l-1}(T_i) (1+\la \wt(T_i) S(T_i))^{-1}\nonumber\\
  		& =& \frac{1}{D\theta}  \sum_{i=1}^D \wt(T_i) M_{l-1}(T_i) \prod_{\substack{1\,\leq\, j\, \leq\, D\\ j\neq i}} (1+\lambda \wt(T_j) S(T_j)),
  	\end{eqnarray*}
  	where in the second equality we used equation~\eqref{eq:S:prod}.  Finally, combining this with (\ref{eq:Msharp and Motimes}) and $M^\sharp(T) \geq M_{l}(T)$, we deduce the conclusion.
  \end{proof}

 \subsection{The tail bound lemma} In the previous sections, we have derived recursive bounds for the survival time $S(T)$ and the number of infections at the leaves $M_l(T)$. These quantities are random variables defined on the probability space of the Galton-Watson tree. In this section, we provide a key lemma that gives tail estimates for $S(T)$ and $M_l(T)$ using the recursions. Let $D\sim \mu$ where $\mu$ is a probability measure on $\N$ satisfying $\E e^{3D}<\infty$. In particular, the measure $\mu$ in Theorem \ref{thm:uper} satisfies this condition.
For the rest of the paper, let $A, B, c$ be any constants that satisfy $A\ll \frac{1}{c}\ll B$. For definiteness, we set
\begin{equation}\label{eq:const:1}
A = \log \E e^{3D}+2e, \quad c = \frac{1}{A},\quad B = 24A.
\end{equation}
We assume that the infection rate $\la$ is sufficiently small compared to $B^{-1}$. Let
\begin{equation}\label{def:q}
q= \frac{c}{16\la^{2}\log (\la^{-1})}
\end{equation}
and
\begin{equation}\label{def:p}
p = (B \la)^{-q} \exp(-c \la^{-2}/B) .
\end{equation}

\begin{definition} \label{def:good:tail}
	A nonnegative random variable $X$ is said to have a good tail (with respect to $A, B, \lambda$) if
	\begin{eqnarray} \label{eq:lm:tail}
	\P(X\ge t) \le f(t):=\begin{cases}
	\exp(-c t\la^{-1}) \quad\text{for } A\la \le t\le \frac{1}{B\la }\\
	p  t^{-q} (\log (eB t\la))^{-2} \quad\text{for } \frac{1}{B\la } \le t.
	\end{cases}
	\end{eqnarray}
	A random variable $X$ is said to have a strong tail if
	\begin{eqnarray} \label{eq:lm:tail:1}
	\P(X\ge t) \le f(3t)/3 \quad\text{for all $t\ge A\la/3$}.
	\end{eqnarray}
	Since $f$ is decreasing, it is clear that a random variable with a strong tail also has a good tail. Moreover, if $X_1, X_2, X_3$ have strong tails then $X_1+X_2+X_3$ has a good tail because
	$$\P(X_1+X_2+X_3\ge t)\le \P(X_1\ge t/3)+\P(X_2\ge t/3)+\P(X_3\ge t/3).$$
\end{definition}
A random variable with a good tail starts with an exponential tail and then a power law tail. The extra term $(\log (eB t\la))^{-2}$ in the latter tail is merely a technical device (used in \eqref{eq:logsum}) for us to prove the following recursive lemma which is essential for our subsequent analysis.
\begin{proposition}[Recursive tail bound] \label{prop:tail} 
Let $D\sim \mu$ where $\mu$ is a probability measure on $\N$ satisfying $\E e^{3D}<\infty$.  There exists $\lambda_0(\mu)>0$ such that the following holds.	Let $\{X_i\}_{i=1}^{D}$ be independent, not necessarily identically distributed, nonnegative random variables with good tails. Then for $0<\lambda<\lambda_0(\mu)$,
	\begin{equation}
	X=  \prod_{i=1}^{D} (1 + \la X_i) - 1,
	\end{equation}
	 has a strong tail.
\end{proposition}
\begin{remark}
	We remark that this lemma works for any $q$ satisfying
\begin{equation}\label{def:q:1}
20\log \frac1\la<q< \frac{c}{8\la^{2}\log (\la^{-1})}.
\end{equation}
Here, we set $q$ as in \eqref{def:q} for clarity.
\end{remark}
The rest of this section is devoted to the proof of Proposition \ref{prop:tail}.

For notational convenience, we let $\La = \frac1\la$. We split the proof into steps depending on the range of $t$ in $\P(X\ge t)$.

\subsubsection{Small $t$: $\frac{A\la}{3} \le t\le \frac{\La}{3B}$}
We need to show that
\begin{equation}\label{eq:smallt}
\P(X\ge t) \le \exp(-3 c t\La)/3= f(3t)/3=:p_t.
\end{equation}
Let $M=t\La\ge A/3$ by \eqref{eq:const:1}.
Since $\E e^{3 D}<\infty$ by \eqref{cond:mu}, we have
\begin{eqnarray}
\P(D\ge M) \le \E e^{3D}\exp\left (-3M\right )\le \exp\left (-3c t\La\right )/24=p_t/8.\nonumber
\end{eqnarray}
So, it suffices to assume that $D\le M$.
The next lemma allows to use truncation on $X_i$.
\begin{lemma} \label{lm:M:X} We have
	$$	\P\left (\exists i\le M: X_i\ge  \frac{2\La}{B}\right ) \le  \exp(-3c t\La)/24=p_t/8.$$
\end{lemma}
\begin{proof}
	Since the $X_i$ have good tails, by the union bound, we have
	\begin{eqnarray}
	\P\left (\exists i\le M: X_i\ge \frac{2\La}{B}\right ) &\le& M2^{-q}\exp(-c \La^2/B) \nonumber\\
	&\le& \La^2 2^{-q}  \exp(-c \La^2/B) \le \frac{1}{24} \exp(-c \La^2/B) \le \exp(-3c t\La)/24,
	\end{eqnarray}
	where we used \eqref{eq:const:1}, \eqref{def:q} and  $t\le \frac{\La}{3B}$.
\end{proof}

Lemma \ref{lm:M:X} allows us to get the tail bound for the sum of $X_i$ as follows.
\begin{lemma}\label{lm:sum:X} We have
	$$\P\left (\sum_{i=1}^{M} X_i \ge (\log 2) \La\right ) \le \exp(-3ct\La)/12=p_t/4.$$
\end{lemma}
\prf Let
$X'_i = X_i\textbf{1}_{X_i\le \frac{2\La}{B}}$.
By Lemma \ref{lm:M:X}, it suffices to show that
$$\P\left (\sum_{i=1}^{M} X'_i \ge (\log 2)\La\right ) \le \exp(-ct\La)/24.$$
Let $\xi_i = \La X'_i\le \frac{2\La^{2}}{B}$ then
\begin{eqnarray}
\E e^{c\xi_i/4} &=& 1+ \frac{c}{4}\int_{0}^{\frac{2\La^2}{B}} e^{cx/4}\P(\xi_i\ge x)dx \nonumber\\
&\le& 1+e + \frac{c}{4}\int_{A}^{\frac{2\La^2}{B}} e^{cx/4}\P(\xi_i\ge x)dx \quad\text{since $cA\le 4$ by \eqref{eq:const:1}}\nonumber\\
& \le&1+ e + \frac{c}{4}\int_{A}^{\frac{\La^2}{B}} e^{-cx/2}dx + \frac{c}{4}\int_{\frac{\La^2}{B}}^{\frac{2\La^{2}}{B}} e^{cx/4} e^{-c\La^2/B}dx\no
& \le&1+ e+2+1\le e^{2}.\label{eq:xi}
\end{eqnarray}
And so
$$\P\left (\sum_{i=1}^{M} \xi_i\ge (\log 2) \La^2\right )\le e^{-c(\log 2) \La^2/4} e^{2M}= \exp\left (2t\La- c (\log 2)\La^2/4\right )\le \exp(-3c t\La)/24,$$
where we used \eqref{eq:const:1} to see that $2t\La + 3ct\La \le \La^{2}(2+c)/B\le \frac{c(\log 2 )\La^{2}}{8}$.
\qed

Combining the above lemmas, we get
\begin{equation}
\P(X\ge t)\le p_t/2+ \P\left (\prod_{i=1}^{M} (1 + \la X_i)-1\ge t,\quad \sum_{i=1}^{M} X_i\le (\log 2) \La,\quad X_i\le \frac{2\La}{B} \text{ for all $i$}\right ).\nonumber
\end{equation}

When $\sum_{i=1}^{M} X_i\le (\log 2) \La$, we have
\begin{eqnarray}\label{eq:sum:prod}
\prod_{i=1}^{M} (1 + \la X_i) \le e^{\la \sum_{i=1}^{M}  X_i} \le 1 + 2\la \sum _{i=1}^{M} X_i,
\end{eqnarray}
where in the second inequality, we used that $e^{s}\le 1+2s$ for all $0\le s\le \log 2$.

Thus,
\begin{equation}
\P(X\ge t)\le p_t/2+ \P\left (2\la\sum_{i=1}^{M} X_i \textbf{1}_{X_i\le \frac{2\La}{B}}\ge t\right ).\nonumber
\end{equation}
%
Let $\xi_i = \La X_i\textbf{1}_{X_i\le \frac{2\La}{B}}$ then as in \eqref{eq:xi},
\begin{eqnarray}
\E e^{c\xi_i/4} &\le&  e^{2}.\nonumber
\end{eqnarray}
and so for $\La$ sufficiently large,
$$\P\left (2\sum_{i=1}^{M} X_i\textbf{1}_{X_i\le \frac{2\La}{B}}\ge t\La \right ) \le e^{-ct \La^2/8} e^{2M} = \exp\left (2t\La-ct \La^2/8\right )\le \exp(-3ct\La)/6=p_t/2.$$
This proves \eqref{eq:smallt}.

\subsubsection{Large $t$: $ t>\frac{\La}{3B}  $.}\label{section:large:t}
We need to show that
\begin{equation}\label{eq:larget}
	\P(X\ge t) \le f(3t)/3=  p (3t)^{-q} (\log (3eB t\la))^{-2}/3=:p_t.
\end{equation}
Firstly, we note that for every $r>0$, by Markov's inequality and the assumption \eqref{cond:mu} on the exponent moment of $\mu$,
\begin{eqnarray}\label{eq:tail:mu}
\P\left (D\ge r\right ) \le \E e^{3 D}e^{-3 r}.
\end{eqnarray}

Let $Y_i = \log(1+\la X_i)$. Then $\log X\le \sum_{i=1}^{D} Y_i$ and by the good tail of $X_i$,
\begin{eqnarray}\label{tail:Y:1}
\P\left (Y_i\ge s\right )\le \begin{cases*}
\exp(-c(e^{s}-1)\La^{2}) \quad\text{if } \log(1+A\la^{2})\le s\le \beta,\\
p \la^{q}(e^{s}-1)^{-q}\log^{-2}(eB(e^{s}-1)) \quad\text{if } s > \beta,
\end{cases*}
\end{eqnarray}
where $\beta = \log\left (1+\frac1B\right )$.


Let $d = \E D$, we get
\begin{eqnarray}
\P(X\ge t) &=& \P(\log X\ge \log t)\le \P\left (\sum_{i=1}^{D} Y_i\ge \log t\right )\nonumber\\
&\le& \P \left (\sum_{i=1}^{2d} Y_i\ge \log t\right ) +\sum_{r = 2d+1}^{\infty} \P \left (\sum_{i=1}^{r} Y_i\ge \log t, D = r\right ).\label{eq:X}
\end{eqnarray}
To handle $\P \left (\sum_{i=1}^{r} Y_i\ge \log t\right )$, we decompose the sum $\sum Y_i$ into subsums in which $Y_i\le \beta$, and $Y>\beta$ and denote the probability that each subsum is large as follows
\begin{eqnarray*}
	p_{1, h, s} &=& \max_{1\le i_1<\dots< i_h\le r} \P\left (\sum_{j=1}^{h} Y_{i_j}\ge s, Y_{i_j}\le \beta\right ) \\
	p_{2, h, s} &=& \max_{1\le i_1<\dots< i_h\le r}  \P\left (\sum_{j=1}^{h} Y_{i_j}\ge s, Y_{i_j}>\beta\right ).
\end{eqnarray*}
Observe that if $\sum_{i=1}^{r} Y_i\ge \log t$ then by setting $s$ to be the integer part of the first subsum, the second subsum must be at least $\log t - s-1$. Thus,
\begin{eqnarray}
\P\left ( \sum_{i=1}^{r} Y_i\ge \log t\right )
&\le& \sum_{k=0}^{r} {r \choose k} \sum_{s \in \N, s \le\log t} p_{1, k, s} p_{2, r - k, \log t - s - 1}. \label{eq:PYr}
\end{eqnarray}

In the next two lemmas, we bound these probabilities.

\begin{lemma}\label{lm:p1}
	We have for every $k, s\in \N$,
	\begin{equation}
	p_{1, k, s} \le  8^{k} e^{-\frac12 c\La^{2}s}. \nonumber
	\end{equation}
\end{lemma}

\begin{proof}
	Let $Y'_i = Y_i \textbf{1}_{Y_i\le \beta}$. Then
	$$\P\left (\sum_{i=1}^{k} Y_i\ge s, Y_i\le \beta\right ) \le \P\left (\sum_{i=1}^{k} Y'_i\ge s\right ).$$
	For a positive number $z$, we have $e^{zY'_i} \le 1+  e^{zY_i} \textbf{1}_{Y_i\le \beta}=1+(1+\la X_i)^{z} \textbf{1}_{X_i\le \La/B}$ and so
	\begin{eqnarray}
	\E e^{zY'_i}
	&\le& 1+(1+A\la^{2})^{z} +\la z\int_{A\la}^{\La /B} (1+\la t)^{z-1} \P(X_i\ge t)dt\nonumber\\
	&\le&1+ (1+A\la^{2})^{z} + \la z\int_{A\la}^{\La /B} (1+\la t)^{z-1} e^{-ct\La}dt\nonumber\\
	&\le& 1+(1+A\la^{2})^{z} + \la z\int_{A\la}^{\La /B}  e^{-t(c\La - \la(z-1))}dt\nonumber.
	\end{eqnarray}
	By choosing $z = \frac{1}{2} c\La^{2}$, we get
	\begin{eqnarray}
	\E e^{zY'_i}  &\le& 1+(1+A\la^{2})^{c\La^{2}/2} + \frac12 c\La \int_{A\la}^{\infty}  e^{-t c\La/2}dt\le 1+2e^{cA/2} \le  8\nonumber,
	\end{eqnarray}
	where we used $cA\le 1$.
	Thus, by Markov's inequality,
	\begin{eqnarray}
	\P\left (\sum_{i=1}^{k} Y'_i\ge s\right )\le e^{-zs} \prod_{i=1}^{k} \E e^{z Y'_i } \le 8^{k} e^{-\frac12 c\La^{2}s}\nonumber
	\end{eqnarray}
	giving the desired bound.
\end{proof}

\begin{lemma}\label{lm:p2}
	For every $m\ge 1$ and $s'>2\beta m$, we have
	\begin{equation}
	p_{2, m, s'} \le \left (\frac{8p 2^{q}\la^{q}}{\beta^{2+q}}\right )^{m}e^{-q (s'-\beta m)}  ( s' - \beta m)^{-2}.\nonumber
	\end{equation}
	
\end{lemma}
\begin{proof}
	We have
	\begin{eqnarray}
	\P\left (\sum_{i=1}^{m} Y_i\ge s', Y_i > \beta\right ) &\le&   \sum_{u \N^m:  \sum_{i=1}^{m} u_i\ge (s'/\beta- m)} \prod_{i=1}^{m} \P(Y_i\in \beta  (u_i, u_{i}+1])\nonumber\\
	&\le & \sum_{u} \prod_{i=1}^{m} \P(Y_i>\beta u_i)\nonumber,
	\end{eqnarray}
	where the last sum runs over all $m$-tuples $(u_1, \dots, u_m)\in \N^{m}$ with $u_i\ge 1$ and $  \sum u_i = s'/\beta - m$ (note that the condition $s'>2\beta m$ implies that such a tuple exists).
	By \eqref{tail:Y:1} (with $s = \beta u_i\ge\beta$), we have
	\begin{eqnarray}
	p_{2, m, s'}\le p^{m} \la^{mq}\sum_{u} \prod_{i=1}^{m} \left (e^{\beta u_i} - 1\right )^{-q} \prod_{i=1}^{m} \log^{-2}(eB (e^{\beta u_i}-1)).\nonumber
	\end{eqnarray}
	We observe since $\beta\leq 1$ we have that for all $x\ge\beta$, we have $e^{x}-1\ge \frac{x}{x+1} e^{x} \ge \frac{\beta}{\beta+1} e^{x}\ge \frac{\beta}{2} e^{x}.$ Thus,
	\begin{eqnarray}
	p_{2, m, s'} &\le& p^{m} \la^{mq}\sum_{u} \prod_{i=1}^{m} \left (\frac{\beta}{2}e^{\beta u_i} \right )^{-q} \prod_{i=1}^{m} \log^{-2}(eB \beta e^{\beta u_i}/2)\nonumber\\
	&=& p^{m} \la^{mq}\left (\frac{2}{\beta} \right )^{qm}   \sum_{u}    e^{-q \beta \sum_{i=1}^{m} u_i}   \prod_{i=1}^{m} \left (\beta u_i + \tau\right )^{-2}\nonumber\\
	&\le & p^{m} \la^{mq}\left (\frac{2} {\beta}\right )^{qm}  e^{-q (s'-\beta m)}   \sum_{u}     \prod_{i=1}^{m} \left (\beta u_i + \tau\right )^{-2}\nonumber,
	\end{eqnarray}
	where $\tau = \log \left (eB \beta/2\right ) \ge \log (2e/5)>\beta$ as $B\ge 32$.
	
	To handle the sum-product, we use the following observation that holds for all integer $w\ge 2$:
	\begin{eqnarray}\label{eq:logsum}
		\sum_{v = 1}^{w-1} (\tau+\beta v)^{-2}(\tau + \beta (w - v))^{-2} &\le& 2\sum _{v = 1}^{w/2} (\tau+\beta v)^{-2}(\tau + \beta w/2)^{-2}\nonumber\\
		& \le& \frac{8}{\beta^{2}(2\tau+\beta w)^{2}}\sum_{v=1}^{\infty} \left (v+\frac\tau\beta\right )^{-2}  \le\frac{8}{\beta \tau(\tau + \beta w)^{2}}.
	\end{eqnarray}
	By induction over $m$, we have that
	\begin{eqnarray}
	\sum_{u}     \prod_{i=1}^{m} \left (\beta u_i + \tau\right )^{-2} = \sum_{u_1+\dots + u_m = s'/\beta  - m}     \prod_{i=1}^{m} \left (\beta u_i + \tau\right )^{-2} \le \frac{8^{m}}{\beta^{m}\tau^{m}} (\tau + s' - \beta m)^{-2}.\nonumber
	\end{eqnarray}
	Thus,
	\begin{equation}
	p_{2, m, s'} \le  \left (\frac{8p 2^{q}\la^{q}}{\beta^{q}\beta \tau }\right )^{m}e^{-q (s'-\beta m)}  (\tau + s' - \beta m)^{-2}\le \left (\frac{8p 2^{q}\la^{q}}{\beta^{q+2} }\right )^{m}e^{-q (s'-\beta m)}  (s' - \beta m)^{-2}.\nonumber
	\end{equation}
\end{proof}

We now put everything together in \eqref{eq:PYr}. We control $\sum Y_i$ by a case analysis writing
$$S_1 = \{s\in \N: s> (\log t)/2\},$$
$$S_2 = S_{2, k, r} =  \{s\in \N: s\le (\log t)/2, s'\le 2(r-k)\beta\},$$ and $$S_3 = S_{3, k, r} =\{s\in \N: s\le (\log t)/2, s'>2(r-k)\beta\},$$
where $s' = \log t - s-1$, and setting
\begin{eqnarray*}
	\P\left ( \sum_{i=1}^{r} Y_i\ge \log t\right )
	&\le& \sum_{k=0}^{r} {r \choose k} \sum_{s\in S_1} p_{1, k, s} +\sum_{k=0}^{r-1} {r \choose k} \sum_{s\in S_2}  p_{2, r - k, s'}\\
	&&\quad+ \sum_{k=0}^{r-1} {r \choose k} \sum_{s\in S_3}  p_{1, k, s}  p_{2, r - k, s'} =: T_1+T_2+T_3.
\end{eqnarray*}
Note that in the sums defining $T_2$ and $T_3$, the index $k$ is at most $r-1$ because otherwise $s$ would have been $[\log t]>\log t/2$.
By Lemma \ref{lm:p1} and $t\ge \La/(3B)$, we have
\begin{eqnarray*}
	T_1 	&\le&  \sum_{k=0}^{r} {r \choose k} \sum_{s\in S_1} 8^{k} e^{-\frac12 c\La^{2}s}\le 2^{r}8^{r}e^{-\frac14 c\La^{2} \log t}\\
	&\le&   16^{r} p (3t)^{-q} (\log (3eB t\la))^{-2}/(9 \La)=16^{r} p_t/(3\La).
\end{eqnarray*}

The contribution from $S_2$ is bounded by
\begin{eqnarray*}
	T_2   &\le&   	\sum_{k=0}^{r-1} {r \choose k} \sum_{s\in S_2}  \max_{1\le i_1<\dots< i_{r-k}\le r} \prod_{j=1}^{r-k}\P\left ( Y_{i_j} > \beta\right )  \\
	&\le&   	\sum_{k=0}^{r-1} {r \choose k} \sum_{s\in S_2} \left (p\la^{q} B^{q}\right )^{r-k} \quad\text{by \eqref{tail:Y:1}}.
\end{eqnarray*}
We note that for $s\in S_2$,
\begin{equation}
2(r - k)\beta \ge s' = \log t - s-1\ge \log t/3.\nonumber
\end{equation}
So,
$$r - k\ge \frac{\log t}{6\beta}\ge \frac{\log t}{3}.$$
Hence,
\begin{equation}
T_2\le 2^{r}|S_2| \left (p\la^{q} B^{q}\right )^{ \log t/3} \le 2^{r} (\log t)\left (p\la^{q} B^{q}\right )^{ \log t/3} \le 16^{r} p_t/(3 \La).\nonumber
\end{equation}

By Lemma \ref{lm:p2}, we obtain
\begin{eqnarray*}
	T_3	&\le&  t^{-q} e^{q }\sum_{k=0}^{r-1} {r \choose k}  8^k \left (\frac{8p 2^{q}\la^{q}e^{q\beta}}{\beta^{q+2} }\right )^{r-k}  \sum_{s \in S_3}  e^{-\frac14 c\La^{2}s}   (s'-\beta(r-k))^{-2}, \end{eqnarray*}
where we used $q\le \frac14 c\La^{2}$ to get $e^{-\frac14 c\La^{2}s}e^{-qs'}\le e^{-q(s+s')}=t^{-q}e^{q}$.
By the definition of $S_3$, we have $s\le  \log t/2$ and so $s'\ge \log t/3$. Since $s'\ge 2(r-k)\beta$, we have $s' - (r-k)\beta \ge  s'/2 \ge \log t/6$. Thus, $ (s'-\beta(r-k))^{-2}\le 36 (\log(3eBt\la))^{-2}$. So
\begin{eqnarray*}
	T_3	&\le&   36 t^{-q} (\log (3eB t\la))^{-2} e^{q } \sum_{k=0}^{r-1} {r \choose k} 8^k\left (\frac{8p 2^{q}\la^{q}e^{q\beta}}{\beta^{q+2} }\right )^{r-k}  \sum_{s=0}^{\infty}  e^{-\frac14 c\La^{2} s} \\
	&\le&  72 t^{-q} (\log (3eB t\la))^{-2} e^{q } 16^r\left (\frac{8p 2^{q}\la^{q}e^{q\beta}}{\beta^{q+2} }\right ) \le 16^{r} p_t /(3\Lambda),
\end{eqnarray*}
where in the second inequality, we used that $\frac{8p 2^{q}\la^{q}e^{q\beta}}{\beta^{q+2} }\le 1$.
So, in total,
\begin{eqnarray*}
	&&	\P\left ( \sum_{i=1}^{r} Y_i\ge \log t\right ) \le   16^{r} p_t/\La.
\end{eqnarray*}
Combining this with \eqref{eq:tail:mu} and \eqref{eq:X}, we obtain
\begin{eqnarray}
\frac{ \P(X\ge t)}{p_t}  &\le&   \frac{16^{2d}}{\La}+\sum_{r = 2d+1}^{\infty} \E e^{3 D}e^{-3r}\frac{16^{r}}{\La}\le 1\nonumber
\end{eqnarray}
for sufficiently small $\la$.
This completes the proof of Proposition \ref{prop:tail}.
 
\subsection{Tail bounds for the excursion time and the total infections at the leaves} Having proved the recursive Proposition \ref{prop:tail}, we are now ready to deduce tail estimates for $S(\T_{\le l})$ and $M(\T_{\le l})$ where $\mathcal{T}\sim \gw(\mu)$.

\subsubsection{Tail bounds for the excursion time}
   \begin{lemma}\label{lm:S tail bd}
 	Let $l\geq 0$ be an integer. Let $\mathcal{T}\sim \gw(\mu)$ and $S(\mathcal{T}_{\le l})$ be the expected survival time on the truncated tree $\mathcal{T}_{\le l}$. Then
 	$ \wt(\T_{\le l}) S(\mathcal{T}_{\le l}) $ has a good tail where $\wt(\T_{\le l})$ is the weight defined in \eqref{def:w:G}.
 \end{lemma}
\begin{proof}
	We prove the lemma by induction on $l$. The case that $l=0$ is trivial because $ \wt(\T_{\le l}) S(\mathcal{T}_{\le l})=\frac{\la}{1+\la}$. Assume that we have the statement for $l-1$.
Let $X = \wt(\T_{\le l}) S(\T_{\le l})$ and $X_i = \wt(\T_i) S(\T_i)$ where $\T_i$ is the subtree of $\T_{\le l}$ containing the $i$-th child of the root of $\T_{\le l}$ and its successors. By Lemma \ref{lm:S:recursion}, we have
\begin{eqnarray*}
	X
	&\le& \frac{-1}{(1+\la (D+1))^{2}}+ \prod_{i=1}^{D} (1 + \la X_i) \le  X_a+   \prod_{i=1}^{D} (1 + \la X_i) - 1,
\end{eqnarray*}
where $X_a= 2\la (D+1)+\left [\frac{\la(D+1)}{1+\la (D+1) }\right ]^{2}$ contains  the constant part.

By the induction hypothesis, $X_i$ has a good tail for all $i$. Thus, by Proposition \ref{prop:tail}, $\prod_{i=1}^{D} (1 + \la X_i) - 1$ has a strong tail. Since $\E e^{9D}<\infty$ by~\eqref{cond:mu}, one can see that both $2\la (D+1)$ and $\left [\frac{\la(D+1)}{1+\la (D+1) }\right ]^{2}$ have strong tails. Thus, the sum of those three has a good tail; so does $X$. This completes the proof of the lemma.
\end{proof}

%

\subsubsection{Tail bounds for the total infections at the leaves}
  \begin{lemma}\label{lm:M tail bd}
	Let $l\geq 0$ be an integer. Let $\mathcal{T} \sim \gw(\mu)$ and $M_{l}(\mathcal{T}_{\le l})$ be the expected total infections at depth-$l$ leaves on $\mathcal{T}_{\le l}$. Then
	$(1.5)^{l} \wt(\T_{\le l}) M_{l}(\mathcal{T}_{\le l}) $ has a good tail.
\end{lemma}

\begin{proof} [Proof of Lemma \ref{lm:M tail bd}] Let
	$$W(\T_{\le l}) = \max\left \{(1.5)^l \wt(\T_{\le l}) M_{l}(\T_{\le l}), \wt(\T_{\le l}) S(\T_{\le l})\right \}.$$
	We shall prove by induction that $W(\T_{\le l})$ has a good tail. This will imply that $(1.5)^l \wt(\T) M_{\le l}(\mathcal{T}_{\le l})$ also has a good tail.
	For $l=0$, we have $W(\T_{\le 0})=\frac{\la}{1+\la}$ which trivially has a good tail. Assume that the statement holds for $l-1$.

	Using the same notations as in the  proof of Lemma \ref{lm:S tail bd}, we have
	\begin{eqnarray*}
		\wt(\T_{\le l}) S(\T_{\le l}) \le  X_a+   \prod_{i=1}^{D} (1 + \la \wt(\T_i) S(\T_i)) - 1 \le X_a+\prod_{i=1}^D (1+2\la W^{l-1}(\T_i))-1,
	\end{eqnarray*}
where we noted that $\T_i$, being the subtree of $\T_{\le l}$, is the same as $(\T_i)_{\le l-1}$.
	By Lemma \ref{lm:Mrecursion},
	\begin{eqnarray*}
		(1.5)^l\wt(\T_{\le l}) M(\T_{\le l})
		&\le& 1.5\la \sum_{i=1}^D W^{l-1}(\T_i) \prod_{\substack{1\,\leq\, j\, \leq\, D\\ j\neq i}} (1+\lambda W^{l-1}(\T_j))\\
		& \le& \frac32\left [\prod_{i=1}^D (1+2\la W^{l-1}(\T_i))-1\right ].
	\end{eqnarray*}
Thus,
$$  W(\T_{\le l})\le X_a+\frac32\left [\prod_{i=1}^D (1+2\la W^{l-1}(\T_i))-1\right ].$$

By induction hypothesis, $W^{l-1}(\T_i)$ has a good tail for all $i$. By  Proposition \ref{prop:tail},  $\prod_{i=1}^D (1+2\la W^{l-1}(\T_i))-1$ has a strong tail.
 Moreover, $X_a$ also has a strong tail.
So, for all $t\ge A\la$,
$$\P(W(\T_{\le l})\ge t)\le \P(\prod_{i=1}^D (1+2\la W^{l-1}(\T_i))-1\ge \frac{t}{3}) +\P(X_a\ge \frac{t}{3})\le   2f(t)/3\le f(t).$$
Thus, $W(\T_{\le l})$ also has a good tail.
\end{proof}

 \section{Proof of Theorem \ref{thm:uper}: upper bound for survival time on graphs}\label{sec:upper}
 As mentioned in Lemma \ref{lem:lwc}, the local neighborhood of a vertex in $G_n\sim \mathcal G(n, \mu)$ is asymptotically a truncated GW tree. In Section \ref{sec:upper:tree}, we have derived upper bounds for the expected survival time and number of infections at the leaves on truncated GW trees. To pass this result to the actual local neighborhoods, we show  in Section \ref{subsec:coupling local nbd} that all local neighborhoods of $G_n$ are contained in the so-called edge-added Galton-Watson trees ($\egw$) which consist of truncated GW trees and cycles ($\gwc$). We then derive similar upper bounds for the expected survival time and number of infections at the leaves on $\gwc$ in Section \ref{sec:unicycle} and on $\egw$ in Section \ref{sec:egw}. The derivation of these bounds are similar (though slightly more complicated) to the bounds for GW trees that we have done in Section \ref{sec:upper:tree}. We will combine all of these ingredients and finish the proof of Theorem \ref{thm:uper} in Section \ref{section:graph}.

 \subsection{Coupling the local neighborhood of $G_n$}\label{subsec:coupling local nbd}
 In this section, we follow \cite{BNNS, NNS} and show the coupling between random graphs and variants of GW trees. We utilize the notion of \textit{augmented distributions}, which allows us to stochastically dominate $N(v,l)$ by a larger geometry.

 \begin{definition}[Definition 4.2, \cite{BNNS}](Augmented distribution)\label{def:aug}
 	Let $\ep\in(0,1)$ and $\mu\equiv \{p_k\}_{k\in\mathbb{N}}$ be a probability measure on $\mathbb{N}$ such that $\E_{D\sim \mu} e^{cD} <\infty$ for some $c>0$.  Let
 	$$k_0 := \max \{k: \sum_{j\geq k} j\sqrt{p_j} \geq \ep/10 \}<\infty\quad\text{and} \quad k_{\textnormal{max}} :=\sup\{k:p_k>0 \}\in \N\cup\{\infty\}.$$
 If $k_0 < k_{\textnormal{max}}$, we define the \textit{augmented distribution} $\mu^\sharp$ of $\mu$ by
 	\begin{equation*}
 	\begin{split}
 	\mu^\sharp (k) := \frac{1}{Z}
 	\begin{cases}
 	p_k/2  &\textnormal{if } k\leq k_0 ;\\
 	\sqrt{p_k} &\textnormal{if } k>k_0,
 	\end{cases}
 	\end{split}
 	\end{equation*}	
 	where $Z= \sum_{k\leq k_0}  p_k/2 + \sum_{k>k_0} \sqrt{p_k}$ is the normalizing constant. When $k_0 = k_{\textnormal{max}}$, we set
 	\begin{equation*}
 	\mu^\sharp (k) := \frac{1}{Z}
 	\begin{cases}
 	p_k/2  &\textnormal{if } k< k_0 ;\\
 	\sqrt{p_k} &\textnormal{if } k=k_0,
 	\end{cases}
 	\end{equation*}
 	for the normalizing constant $Z= \sum_{k<k_0} p_k/2 + \sqrt{p_{k_0}}$.
 \end{definition}

 Here are some basic properties of augmented distributions.

 \begin{lemma}[Lemma 4.3, \cite{BNNS}]\label{lem:aug}
 	Let $\mu$ be a probability distribution on $\mathbb{N}$.
 	\begin{enumerate}
 		\item[(1)] If $\mu$ has an exponential tail, then so does $\mu^\sharp$. More specifically, if $\E _{D\sim \mu} e^{3cD}<\infty$ then $\E _{D\sim \mu^\sharp} e^{cD}<\infty$.
 		
 		\item[(2)] 	Let $D_1,\ldots,D_n$ be $n$ independent samples of $\mu$. For a subset $\Delta \subset [n]$, let $\{p_k^\Delta \}_k$ denote the empirical distribution of $\{D_i \}_{i\in [n]\setminus \Delta}$. With high probability over the choice of $D_i$'s, $\{p_k^\Delta \}_k $ is stochastically dominated by $\mu^\sharp,$ for any $\Delta \in [n]$ with $|\Delta| \leq n/3$.
 	\end{enumerate}
 \end{lemma}

 \begin{remark}
 	The i.i.d. $D_i $ in Lemma \ref{lem:aug} can be viewed as a degree sequence of $G_n \sim \mathcal{G}(n,\mu)$. Consider the exploration procedure starting from a single fixed vertex $v$, which, at each step, reveals a vertex adjacent to the current explored neighborhood and the half-edges incident to the new vertex. Then the second statement says that when the exploration process revealed $N \leq n/3$ vertices inside the local neighborhood of $v$, the empirical degree distribution of the $n-N$ unexplored vertices is stochastically dominated by $\mu^\sharp$, with high probability.
 \end{remark}

 Using the above properties of augmented distributions, we develop a coupling argument to dominate $N(v,L) \subset G_n$ where $G_n\sim \mathcal G(n, \mu)$ by a Galton-Watson type branching process. To this end, we first take account of the effect of emerging cycles in $N(v,L)$.

 For a constant $\gamma>0$, let $\mathcal{A}(\gamma)$ be the event that $N(v,\gamma \log n)$ in $G_n$ contains at most one cycle   for all $ v\in G_n$. The following lemma shows that we typically have $\mathcal{A}(\gamma)$ for some constant $\gamma$.

 \begin{lemma}[Lemma 4.5, \cite{BNNS}]\label{lem:1cyc}
 	There exists $\gamma=\gamma(\mu)>0$ such that for $G_n\sim \mathcal{G}(n,\mu)$,
 	$$\P(G\in \mathcal{A(\gamma)})  = 1-o(1).$$
 \end{lemma}

 We now define the three variants of GW trees that will be used. We use the definitions from \cite[Section 6.1]{NNS}.

 \begin{definition}[$\gwc^{1}$]
 	Let $\xi$ be a positive, integer-valued random variable, and let $m,l\geq 1$ be nonnegative integers. Then, $\mathcal{H}^{m,l} \sim \gwc^1(\xi, m)_{\le l}$, the \textit{Galton-Watson-on-cycle process of type one} (in short, $\gwc^1$-\textit{process}), is generated according to the following procedure:
 	\begin{enumerate}
 		\item [1.] Consider a length-$m$ cycle $C= v_1 v_2 \ldots v_m v_1$.
 		
 		\item[2.] At each $v_j$ for $j=1,\ldots,m$, attach $\mathcal{T}_{j}\sim$ i.i.d.$\; \gw(\xi)_{\le l}$ by setting $v_j$ as its root.
 	\end{enumerate}
 	We designate vertex $v_1$ as the root of $\mathcal{H}^{m,l}$ and denote $\rho = v_1$. Note that $m=1$ corresponds to the usual Galton-Watson trees.
 \end{definition}

 Next, we introduce the \textit{Galton-Watson-on-cycle process of type two} (in short, $\gwc^2$-process), which can
 be thought as a certain subgraph of $\gwc^1$-processes.
 \begin{definition}[$\gwc^{2}$]\label{def:gwc2}
 	Let $\xi$ be a positive, integer-valued random variable, and let $m,l\geq 1$ be integers. Then, $\dot{\cH}^{m,l} \sim \gwc^2(\xi, m)_{\le l}$, the \textit{Galton-Watson-on-cycle process of type two} (in short, $\gwc^2$-\textit{process}), is generated according to the following procedure:
 	\begin{enumerate}
 		\item [1.] Consider a length-$m$ cycle $C= v_1 v_2 \ldots v_m v_1$.
 		
 		\item[2.] At each $v_j ,\; j\in\{2,\ldots,m\}$, attach $\mathcal{T}_{j}\sim$ i.i.d.$\; \gw(\xi)_{\le l}$ by setting $v_j$ as its root. At $v_1$, we do nothing.
 	\end{enumerate}
 	We designate vertex $v_1$ as the root of $\dot{\cH}^{m,l}$ and denote $\rho = v_1$.
 \end{definition}

 We now introduce the Galton-Watson type process which will be coupled with the local neighborhoods of $G_n$.

 \begin{definition}[$\egw$]
 	Let $l,L, m$ be positive integers with $m\geq 2$ and $l\leq L$, and let $\xi$ be a probability distribution on $\mathbb{N}$. We define  $\egw(\xi;l,m)_{\le L}$, the edge-added Galton-Watson process (in short, $\egw$-process) as follows:
 	\begin{enumerate}
 		\item Generate a $\gw(\xi)_{\le L}$ tree, conditioned on survival until depth $l$.
 		
 		\item At each vertex $u$ at depth $l$, add an independent $\gwc^2(\xi,m)_{\le L-l}$ process (see Definition \ref{def:gwc2}) rooted at $u$. Here we preserve the existing subtrees at $u$ which comes from $\gw(\xi)_{\le L}$ tree from the above step.
 	\end{enumerate}
 	Let $\xi'$ be another probability measure on $\mathbb{N}$. Then $\egw(\xi,\xi';l,m)_{\le L}$ denotes the $\egw$-process where the root has degree distribution $\xi$, and all the descendants have $\xi'$. And in the second step of the definition, we add $\gwc^{2}(\xi';m)_{\le L-l}$ instead.
 \end{definition}

 The next step is to show that the local neighborhood $N(v,L)$ is dominated by a combined law of the above graphs. In what follows, for two probability measures $\nu_1$ and $\nu_2$ on graphs, we say $\nu_1$ \textit{stochastically dominates} $\nu_2$ and write $\nu_1 \geq_{st} \nu_2$ if there exists a coupling between $S_1 \sim \nu_1$ and $S_2 \sim \nu_2$ such that $S_2 \subset S_1 $ in terms of isomorphic embeddings of graphs, i.e., there exists an injective graph homomorphism from $S_2$ into $S_1$.

 Fix a vertex $v \in G_n$, and consider its local neighborhood $N(v, \gamma \log n)$ with $\gamma$  as in Lemma \ref{lem:1cyc}. Let $\mathcal{A}=\mathcal{A}(\gamma)$ be the event in the same lemma.
  For each $l,s$ with $s\geq 2$, we define the event $\mathcal{B}_{l,s}$ to be the intersection of $ \mathcal{A}$ and the event that $N(v, \gamma \log n)$ forms a cycle of length $s$ at distance $l$ from $v$.

 For the given degree distribution $\mu$, let $\tilde \mu$ be its size-biased distribution (see \eqref{eq:def:sizebiased} for its definition), and $\widetilde{\mu}' := \tilde \mu_{[1, \infty)}$ denote the distribution $\tilde \mu$ conditioned on being in the interval $[1, \infty)$. Let $\mu^\sharp$ and $\widetilde{\mu}^\sharp$ be the augmented distributions of $\mu$ and $\widetilde{\mu}'$, respectively. Further, let $\eta$, $\eta_{l,s}$ and  $\eta_0$ denote the probability measures on rooted graphs describing the laws of $N(v, \gamma \log n)$, $\egw(\mu^\sharp, \widetilde{\mu}^\sharp;l,s)_{\le \gamma \log n}$ and $\gw(\mu^\sharp, \widetilde{\mu}^\sharp)_{\le \gamma \log n}$, respectively.

 \begin{lemma}[The coupling, Lemma 4.8, \cite{BNNS}]\label{lm:nbdcoupling} 
 	Under the above setting, we have the following stochastic domination:
 	\begin{equation*}
 	\eta \one_{\mathcal{A}} \leq_{st} \sum_{l,s:s\geq 2} b_{s,l} \eta_{s,l} + b_0 \eta_0,
 	\end{equation*}
 	where $b_{l,s}=\P(\mathcal{B}_{l,s})$, $b_0 = 1-\sum_{l,s} b_{s,l}$.
 \end{lemma}

\subsection{Recursive analysis for unicyclic graphs $\gwc^{1}$ and $\gwc^{2}$}\label{sec:unicycle}

In this section, we derive results similar to Section \ref{sec:upper:tree} for $\gwc^{1}$ and $\gwc^{2}$. Let us start with the definitions of survival time and number of infection at the leaves, first for $\gwc^{1}$  in Definition \ref{def:SMunicyclic} and then for $\gwc^{2}$  in Definition \ref{def:SMdef gwc2}.

\begin{definition}[Root-added contact process on $\gwc^1$-processes]\label{def:SMunicyclic}
	Let $m,l\geq 1$ be integers, and $H$ be a graph that consists of a length-$m$ cycle $C=v_1v_2\ldots v_m v_1$ and the trees $T_1,\ldots,T_m$ rooted at $v_1,\ldots,v_m$, respectively.  Consider the \textit{root-added starred contact process} $(X_t)$ on $H$ which is the starred contact process on the graph $H\cup{v_1^+}$ with the permanently infected parent $v_1^+$ having a single connection with $v_1$. The vertex $v_1$ is initially at state $1^*$.
	\begin{enumerate}
		\item [1.]  Let $\textbf{S}(H)$ be the excursion  time which is the first time when $H$ returns to the all-healthy state. $S(H)=\E_{\textsc{cp}} \textbf{S}(H)$  denotes  the expected excursion time.
		
		\vspace{2mm}
		\item [3.] Let $\mathcal{L}_j = \{v\in T_j : \textnormal{dist}(v,v_j)=l \}$. For $v\in \mathcal{L}_j$ for some $j$, we define \textit{the total infections at} $v$ by
		\begin{equation*}
		\begin{split}
		\textbf{M}_{l, v}(H)&:=\left|\left\{s\in [0,\textbf{S}(H)]:\; X_s(v)=1^* \textnormal{ and } X_{s-}(v)=0 \right\}\right|.
		\end{split}
		\end{equation*}
		Then, the total infections at depth-l leaves and its expectation are given as
		\begin{equation*}
		\begin{split}
		\textbf{M}_{l}(H) &:= \sum_{j=1}^m \sum_{v\in\mathcal{L}_j} \textbf{M}_{l, v}(H), \quad \textnormal{and} \quad M_{l}(H) := \E_{\textsc{cp}}\textbf{M}_{l}(H) .
		\end{split}
		\end{equation*}
	\end{enumerate}
\end{definition}

\begin{definition}[Root-added contact process on $\gwc^2$-processes]\label{def:SMdef gwc2}
	For $\dot{\cH}^{m,l} \sim \gwc^2(\xi, m)_{\le l}$, we define the \textit{root-added (starred) contact process} on $\dot{\cH}^{m,l}$ without adding a new parent to the root. Instead, we fix the root $\rho=v_1$ to be permanently infected. Assume that $v_2$ is initially at state $1^*$. Then, we define  $S_2(\dot{\cH}^{m,l})$, $M_{l, 2}(\dot{\cH}^{m,l})$ to be the expected excursion time and the expected total number of infections at depth-$l$ leaves as in Definition \ref{def:SMunicyclic}. If $v_m$ is initially at state $1^*$ instead, the corresponding identities are denoted by $S_m(\dot{\cH}^{m,l})$, $M_{l, m}(\dot{\cH}^{m,l})$.
	
	We also write
	\begin{equation*}
	S(\dot{\cH}^{m,l}) = \frac{1}{\wt(\dot{\cH}^{m, l})} \left (w_2(\dot{\cH}^{m, l}) S_2(\dot{\cH}^{m,l})+ w_m(\dot{\cH}^{m, l}) S_m(\dot{\cH}^{m,l}) \right ), \quad
	\end{equation*}
	and
	\begin{equation*}
	M_{l}(\dot{\cH}^{m,l}) = \frac{1}{\wt(\dot{\cH}^{m, l})}  \left (w_2(\dot{\cH}^{m, l}) M_{l, 2}(\dot{\cH}^{m,l})+ w_m(\dot{\cH}^{m, l}) M_{l, m}(\dot{\cH}^{m,l})\right ),
	\end{equation*}
	where we define the weights to be
	\begin{equation}\label{eq:wcH}
	w_2(\dot{\cH}^{m, l})= \frac{  \la(D_2+2)}{1+\la(D_2+2)}, \quad w_m(\dot{\cH}^{m, l})=\frac{ \la (D_m+2)}{1+\la(D_m+2)}, \quad \wt(\dot{\cH}^{m, l})= w_2(\dot{\cH}^{m, l})+w_m(\dot{\cH}^{m, l}),
	\end{equation}
with $D_j$ being the number of children of $v_j$ in $\T_j$.
\end{definition}

The goal of this section is to establish the following analog of Lemmas \ref{lm:S tail bd} and \ref{lm:M tail bd}.

\begin{lemma}\label{lm:cyclic}
Assume that $ \E e^{3\xi}<\infty$.	Let $m,l\geq 1$ be  integers and $\mathcal{H}^{m,l}\sim \gwc^1(\xi,m)_{\le l}$ and $\dot{\cH}^{m,l}\sim \gwc^2(\xi,m)_{\le l}$. The following random variables have good tails:
	$$\wt({\cH}^{m, l}) S(\mathcal{H}^{m,l}),\  \wt(\dot{\cH}^{m, l}) S(\dot{\cH}^{m,l}),\ (1.5)^{l} 	\wt({\cH}^{m, l})  M_{ l}(\mathcal{H}^{m,l})\quad\text{and}\quad (1.5)^{l-1}\wt(\dot{\cH}^{m, l}) M_{ l}(\dot{\cH}^{m,l}),$$
	where the constants $A, B, c$ in the definition of goodness (Definition \ref{def:good:tail}) are any constants satisfying \eqref{eq:const:1} with $\xi$ in place of $D$ and where
	\begin{equation}
	\wt({\cH}^{m, l}) = \frac{\la (D_1+3)}{1+\la (D_1+3)}.\nonumber
	\end{equation}
\end{lemma}

 Though the proof of Lemma \ref{lm:cyclic} is technically more complicated, the idea is similar to the derivation of Lemmas \ref{lm:S tail bd} and \ref{lm:M tail bd}. Therefore, we defer its proof to Appendix A.

 \subsection{Recursive analysis for $\egw$} \label{sec:egw}
We now prove similar results for \egw-processes. We start with the definitions of survival time and total number of infections.
\begin{definition}[Excursion time and total infections at leaves for \egw]\label{def:SMegw}
	Let $h, l, m$ be nonnegative integers with $m\geq 2$, $\nu, \nu'$ be probability measures on $\mathbb{N}$, and $\mathcal{H}\sim \egw(\nu,\nu';h, m)_{\le l}$. We connect a permanently infected parent $\rho^+$ to the root $\rho$ of $\mathcal{H}$. Consider the root-added starred contact process $(X_t)$ on $\cH$ with $\rho^+$ always infected and $\rho$ initially at $1^*$.
	\begin{itemize}
		\item   The \textit{excursion time} $\textbf{S}(\mathcal{H})$ is the first time when $(X_t)$ becomes all-healthy on $\mathcal{H}$, and  $S(\mathcal{H}) = \E_{\textsc{cp}} \textbf{S}(\mathcal{H})$ denotes the \textit{expected excursion time} on $\mathcal{H}$.
		
		\item Let $\mathcal{L}$ be the collection of \textit{bottom leaves} of $\mathcal{H}$, that is, denoting $\{C_j\}$ to be the length-$m$ cycles in $\mathcal{H}$,
		\begin{equation*}
			\mathcal{L}= \{v\in\mathcal{H}: \textnormal{dist}(v,\rho)\geq l \textnormal{ and dist}(v,C_j)\geq l-h \textnormal{ for all }j \}.
		\end{equation*}
		
		\item Let $v\in \mathcal{L}$ and denote	the \textit{total infections at} $v$ by
		\begin{equation*}
			\textbf{M}_{l, v}(\mathcal{H}):=
			\left|\left\{s \in [0, \textbf{S}(\mathcal{H})]:\;
			X_s(v)=1^* \textnormal{ and } X_{s-}(v)=0
			\right\}\right|.
		\end{equation*}
		Then, we set  the \textit{total infections  at depth-$l$ leaves}  as
		\begin{equation*}
			\textbf{M}_{l}(\mathcal{H}) := \sum_{v\in\mathcal{L}} \textbf{M}_{l, v}(\mathcal{H}).
		\end{equation*}

		\item We also let $M_{l}(\mathcal{H}):=\E_{\textsc{cp}} \textbf{M}_{l}(\mathcal{H})$ be the \textit{expected total infections at leaves}.
	\end{itemize}
\end{definition}

For an \egw-process $\mathcal{H}$, the tails of $S(\mathcal{H})$ and $M_{ l}(\mathcal{H})$ are good, up to appropriate rescales.

\begin{proposition}\label{prop:egw tail}
	Let $\mu$ be a probability measure satisfying the hypothesis of Theorem \ref{thm:uper}. Let $m,l,h$ be nonnegative integers such that $m\geq 2$ and $l\geq h+1$. Let $\mathcal{H}\sim \egw(\mu^\sharp, \widetilde{\mu}^\sharp;h,m)_{\le l}$. Then, it holds that $\wt(\cH) S(\mathcal{H})$ and $1.5^{l}\wt(\cH) M_{ l}(\mathcal{H})$ have good tails where $\wt(\cH) = \frac{\la(D+1)}{1+\la (D+1)}$ and $D$ is the number of children of the root.
\end{proposition}

\begin{proof} Let $\mathcal{H}'\sim \egw( \widetilde{\mu}^\sharp;h,m)_{\le l}$. We shall establish the same conclusion for both $\cH$ and $\cH'$ by an induction on $h$. When $h=0$, since $\cH'\sim \gwc^1 (\widetilde{\mu}^\sharp,m)_{\le l}$ and $\cH\sim \gwc^1(\mu^\sharp,\widetilde{\mu}^\sharp,m)_{\le l}$, the $\gwc^1$-process in which only the root of the Galton-Watson tree has offspring $\mu^\sharp$ and all others have offspring $\widetilde{\mu}^\sharp$.
	Therefore, Lemma \ref{lm:cyclic} gives the desired statement for both $\egw(\widetilde{\mu}^\sharp;0,m)_{\le l}$ and $\egw(\mu^\sharp,\widetilde{\mu}^\sharp;0,m)_{\le l}$ for all $m\geq 2$, $l\geq 0$.
	
	Suppose that we have the statement for $\mathcal{H}'\sim \egw(\widetilde{\mu}^\sharp;h,m)_{\le l}$ with $m\geq 2$, $h\geq 0$ and $l\geq h+1$. Let $\rho_1$ be the root of $\mathcal{H}_1\sim \egw(\widetilde{\mu}^\sharp;h+1,m)_{\le l}$ and $D_1\sim \widetilde{\mu}^\sharp$ be its degree. Then the subgraphs $H_1^1,\ldots , H_{D_1}^1$  rooted at each child of $\rho_1$  are i.i.d. $\egw(\widetilde{\mu}^\sharp;h,m)_{\le l-1}$, and hence, by the induction hypothesis, the desired tail statement holds for $H_1^1,\ldots,H_{D_1}^1$. The argument in proving Lemma \ref{lm:S:recursion} applies verbatim to prove the same recursive bound with $S(\cH_1)$ playing the role of $S(T)$ and the $S(H_1^1), \dots, S(H_{D_1}^1)$ playing the role of the $S(T_i)$. Likewise, Lemma \ref{lm:Mrecursion} applies to get a recursive bound for $M_{l}(\cH_1)$ in terms of $M_{l-1}(H_{1}^1), \dots, M_{l-1}(H_{D_1}^1)$ and $S(H_{1}^1), \dots, S(H_{D_1}^1)$.

	This allows us to apply Proposition \ref{prop:tail}, using an argument identical to that in the proof of Lemmas \ref{lm:S tail bd} and \ref{lm:M tail bd}, to conclude that $\mathcal{H}_1$ satisfies the tail statement.

	The result for $\mathcal{H}_2 \sim \egw(\mu^\sharp,\widetilde{\mu}^\sharp;h+1,m)_{\le l}$ follows similarly as $\mathcal{H}_2$ has subgraphs $H^2_1,\ldots, H^2_{D_2} \sim$ i.i.d. $\egw(\widetilde{\mu}^\sharp; h,m)_{\le l-1}$ with $D_2 \sim \mu^\sharp.$ We note that by Condition \ref{cond:mu} and the first part of Lemma \ref{lem:aug}, $\mu^\sharp$ and $\widetilde{\mu}^{\sharp}$ satisfy the hypothesis of Proposition \ref{prop:tail} which allows the above arguments to go through.
\end{proof}

   \subsection{Finishing the proof of Theorem \ref{thm:uper}}\label{section:graph}
     We are now ready to finish the proof of Theorem \ref{thm:uper}.
 It suffices to show that with high probability on the randomness of $G_n$, for all vertices $v\in V(G_n)$, with probability at least $1 - o(\frac{1}{n})$, the starred contact process $(X^{(v)}_t)$ with $v$ being strongly infected initially dies out before time $n^{C\la^{2}\log(\la^{-1})}$. We then take the union bound over all $v\in V(G_n)$ and note that by using the graphical representation, the survival time of the contact process starting with all vertices infected is the maximum of the survival times of the individual $(X^{(v)}_t)$ to obtain the statement of the theorem (we also use \eqref{eq:unstar} to go from the starred contact process to the original contact process).

 \subsubsection{Handling individual local pieces}
 Recall the definition of $q$ in \eqref{def:q} as
  $$q= \frac{c}{16\la^{2}\log (\la^{-1})}.$$
  Let $\delta = \frac{10}{q}, r = \delta \log n.$  We note that for $\la$ sufficiently small (depending on $\mu$), $\delta$ is also sufficiently small. In particular, we assume $\lambda$ is small enough so that $\delta\le \gamma/100$, the constant appearing in Lemma \ref{lem:1cyc}. By this lemma, with high probability, $N(v, r)$ contains at most one cycle for all $v\in V(G_n)$.
 For each $v$, consider the local neighborhood $B_v$ of $v$ as follows.
 \begin{itemize}
 	\item $B_v := N(v, r)$, if $N(v, r)$ is a tree.
 	
 	\item If $N(v, r)$ contains the unique cycle $C_v$ at distance $h$ from $v$, then
 	$$B_v:= N(v, r)  \cup \left[\bigcup_{u\in C_v}  N(u, r -h) \right]. $$
 \end{itemize}
 We define $B_v^+$ the same way except replacing $r$ with $10r$.  With this definition, $B_v \subset B_u^+$ for all $u\in B_v$. We first provide a bound on the size of $B_v^+$.

 \begin{lemma}\label{lm:Bv}
 	With high probability, for all $v$, it holds that
 	$$|B_v^+|\le n^{C\delta}/2,$$
 	where $C$ is some constant depending on $\mu$.
 \end{lemma}
For the rest of this section, we assume (wlog) that $C>2$.
To prove the lemma, we follow \cite[Theorem 2.3]{allan2009} and will show the following in Appendix \ref{app:treebound}.
\begin{lemma}\label{lm:tree:bound}
	Let 	$\T\sim \gw(\mu, \widetilde{\mu})$ where $\mu$ has an exponential tail $\E_{N\sim \mu} e^{cN}<\infty$ for some $c>0$. Then there exists $t>0$ such that
	\begin{equation}\label{eq:tree:bound}
	\sup_{s}\E \exp(t Z_s \tilde d^{-s})<\infty,
	\end{equation}
	where $\tilde d = \E_{D\sim \widetilde \mu} D$ and $Z_s$ is the number of vertices at level $s$ of $\T$.
\end{lemma}
 \begin{proof}[Proof of Lemma \ref{lm:Bv}] We recall the notion of augmented distributions $\mu^\sharp$ and $\widetilde \mu^\sharp$ before the statement of Lemma \ref{lm:nbdcoupling}. For any vertex $v$, we explore its local neighborhood where at each step, we reveal a vertex adjacent to the current explored neighborhood and the half-edges incident to the new vertex. The second part of Lemma \ref{lem:aug} states that when the exploration process
revealed $N\le n/3$ vertices inside the local neighborhood of $v$, the empirical degree distribution of the $n-N$ unexplored vertices is stochastically dominated by $\mu^\sharp$, with high probability. 
 	Therefore, $|B_v^+|$ is stochastically dominated by the number of vertices in $\T_{\le 20r}$ where $\T\sim \gw(\mu^{\sharp}, \widetilde{\mu^{\sharp}})$ as long as the latter has less than $n/3$ vertices. By the first part of Lemma \ref{lem:aug}, the measure $\mu^{\sharp}$ has an exponential tail. Hence, applying Lemma \ref{lm:tree:bound} to this distribution together with Markov's inequality, with probability $1-  O(\log n\cdot n^{-C})$ for all $0\le s\le 20r$,
 	$$Z_s\le t^{-1} C\tilde d^{s} \log n.$$
 	On this event,
 	$$|B_v^+|\le \sum_{s=1}^{20r} Z_s \le t^{-1}C(\log n)^{2}\cdot  n^{20\delta \log \widetilde d}\le n^{C\delta}/2.$$
 	Taking the union bound over $v$, we obtain the result.
 \end{proof}

 For the contact process $(X_t^{(v)})$ on $G_n$ and for $A\subset V(G_n)$, let $(X^{(v)}_{A}(t))$   be the contact process restricted on $A$, that is we only use the recovery and infection clocks inside $A$. Let $S_v$  be the expected survival time of the contact process $(X^{(v)}_{B_v^+}(t))$. Under $(X^{(v)}_{B_v}(t))$, let $M_v$ be the expected total infections at the leaves $\mathcal{L}(B_v)$ which is defined by
 \begin{equation}\label{eq:def blockbottomleaves}
 \mathcal{L}(B_v) := \begin{cases}
 \{u\in B_v : \textnormal{dist}(u,v) = r \} &\textnormal{if } B_v \textnormal{ is a tree},\\
 \{u\in B_v : \textnormal{dist}(u,v) \geq r \textnormal{ and } \textnormal{dist}(u,C_v) \geq r-h \} &\textnormal{if } C_v \textnormal{ exists},
 \end{cases}
 \end{equation}
 where $C_v$ is the unique cycle (if exists) in $N(v, r)$ and $h =\textnormal{dist}(v,C_v)$. Note that this definition is consistent with the definition of leaves of \egw-processes.
 \begin{lemma}\label{lm:Bv:Mv} We have that
 	\begin{equation}
 		\P\left [\max_{v\in V(G_n)} S_{v}\le n^{\delta}, \max_{v\in V(G_n)}\quad M_{v}\le\frac{1}{\log n}\right ]\geq 1-o(1).
 	\end{equation}
 \end{lemma}

 \begin{proof} Fix $v\in V(G_n)$. By Lemma \ref{lm:nbdcoupling}, $B_v$ and $B_v^+$ are dominated by $\egw(\mu^\sharp, \widetilde{\mu}^\sharp;l,s)_{\le \delta \log n}$ and $\gw(\mu^\sharp, \widetilde{\mu}^\sharp)_{\le \delta \log n}$. By Proposition \ref{prop:egw tail}, it holds that $\frac{\la(D+1)}{1+\la(D+1)} S_v$ and $1.5^{\delta \log n}\frac{\la(D+1)}{1+\la(D+1)} M_v$ have good tails where $D=\deg_{G_n}(v) \sim \mu$. Thus,
 	$$\P_{\text{\tiny RG}}(S_{v}>n^{\delta})\le \P_{\text{\tiny RG}}\left (\frac{\la(D+1)}{1+\la(D+1)} S_v\ge \la n^{\delta} \right ).$$
 	For $n$ sufficiently large, we have $\la n^{\delta}\ge \frac{1}{B\la}$ and so for sufficiently large $n$,
 	$$\P_{\text{\tiny RG}}(S_{v}>n^{\delta})\le C_{\la, \mu} n^{-\delta q+o(1)} \leq n^{-3/2},$$
 		where $C_{\la, \mu}$ is some constant depending on $\mu$ and $\la$.
 	Similarly,
 	$$\P_{\text{\tiny RG}}\left (M_{v}>\frac{1}{\log n}\right ) \le \P_{\text{\tiny RG}}\left (1.5^{\delta \log n}\frac{\la(D+1)}{1+\la(D+1)} M_v>\la n^{\delta/4}\right )\le C_{\la, \mu} n^{-q\delta/4} \leq C_{\la, \mu} n^{-2}.$$	
 	Taking the union bound over $n$ vertices of $G_n$, we obtain the conclusion.
 \end{proof}

Let $\mathcal E_n$ be the intersection of the good events in Lemmas \ref{lm:Bv} and \ref{lm:Bv:Mv}. By Markov's inequality and Lemma \ref{lm:Bv:Mv}, we obtain the following.

 \begin{lemma}\label{lm:Bv:Mv:1} Let $G_n\in \mathcal E_{n}$. Let $C$ be the constant in Lemma \ref{lm:Bv}. For each vertex $v\in V(G_n)$, let $\bs_v$ be the survival time of $(X^{(v)}_{B_v})(t)$ and $\bm_v$ be the number of infections at the leaves of $B_v$ of $(X^{(v)}_{B_v})(t)$. Then
 	\begin{equation}\label{eq:bv}
 		\P_{\CP}(\bs_v \ge jn^{(C+1)\delta}) \le 2^{-j},\quad\text{for all } j\ge 1
 	\end{equation}
 	and
 	 	\begin{equation}\label{eq:mv}
 	\P_{\CP}(\bm_v \ge jn^{(3C+1)\delta}) \le 2^{-j+1},\quad\text{for all } 1\le j\le C\log n.
 	\end{equation}
 \end{lemma}

 \begin{proof}
 	To see \eqref{eq:bv}, at any time $t\le \bs_v$, the probability that $(X^{v}_{B_v})$ survives until time $t + n^{C\delta}$ is bounded as follows.
 	\begin{eqnarray*}
 		\P_{\CP}(\bs_v >t+n^{(C+1)\delta}\vert \bs_v \ge t)  &\le& \sum_{u\in B_v} \P_{\CP}(\bs_v >n^{(C+1)\delta}\vert X^{(v)}_{B_v} \text{ starts with $u$ initially at $1^*$}) \\
 		&\le&   |B_v| n^{-C\delta}\le 1/2,
 	\end{eqnarray*}
 	where in the second inequality we use Markov's inequality, Lemma \ref{lm:Bv:Mv} and the fact that $B_u^+$ contains $B_v$. Applying this probability for $t = k n^{(C+1)\delta}$, $k=0, \dots j-1$, we obtain \eqref{eq:bv}.

 	Next, for \eqref{eq:mv}, let ${\bf C}_{v, t_0}$ be the total number of infection clock rings in $B_v$ before time $t_0$. We observe that $\bm_v\le {\bf C}_{v, \bs_v}$.
 	Thus,
 	\begin{eqnarray}\label{eq:bmv}
 		&& \P_{\CP}(\bm_v \ge jn^{(3C+1)\delta}) \le \P_{\CP}(\bs_v \ge jn^{(C+1)\delta}) + \P_{\CP}({\bf C}_{v, jn^{(C+1)\delta}}\ge  jn^{(3C+1)\delta}).
 	\end{eqnarray}
 	By \eqref{eq:bv}, the first term is bounded by $2^{-j}$. Since $|B_v|\le n^{C\delta}$, ${\bf C}_{v, jn^{(C+1)\delta}}\le \Pois(\la j n^{(3C+1)\delta})$. Thus, by the tail bound \eqref{eq:poisson} for Poisson distributions, the last term in \eqref{eq:bmv} is also bounded by $2^{-j}$.
 \end{proof}

 \begin{lemma} \label{lm:Bv:Mv:2} Let $k=n ^{(6C+6)\delta}$. Let $\bm_1, \dots, \bm_k$ be independent random variables satisfying \eqref{eq:mv} and $\E \bm_i \le \frac{1}{\log n}$. Then
 	\begin{eqnarray*}
 		\P\left (\sum_{i=1}^{k} \bm_i\ge k-1\right )\le n^{-C/2}.
 	\end{eqnarray*}
 \end{lemma}
 \begin{proof}
 	By \eqref{eq:mv},
 	\begin{equation}
 		\P(\bm_i \ge C\log n \cdot n^{(3C+1)\delta}) \le 2n^{-C}.\nonumber
 	\end{equation}
 	And so, as $\delta$ is sufficiently small,
 	\begin{equation}
 		\P(\bm_i \le C\log n \cdot n^{(3C+1)\delta}, \ i=1, \dots, k) \ge 1 - 2kn^{-C}\ge 1 - \frac{1}{2n^{C/2}}.\nonumber
 	\end{equation}
 	Applying the Azuma's inequality to the random variables $X_i : = \bm_i\textbf{1}_{\bm_i \le C\log n \cdot n^{(3C+1)\delta}}$ with $\E X_i\le \frac{1}{\log n}$, we get
 	\begin{eqnarray*}
 		\P(\sum_{i=1}^{k} X_i\ge k-1)\le \exp\left (-\frac{k^{2}}{8\sum_{i=1}^{k}||X_i||^{2}_{\infty}}\right )\le \exp\left (-\frac{k}{8C^{2}(\log n)^{2} \cdot n^{(6C+2)\delta}}\right ) \le \frac{1}{2n^{C/2}},
 	\end{eqnarray*}
where $||X_i||_{\infty} := \inf\{x\in \R: X_i\le x \text{\ a.s.}\}\le C\log n \cdot n^{(3C+1)\delta}$, proving the lemma.
 \end{proof}

\subsubsection{Combining the local pieces}
Finally, we show how to combine the local information to bound the survival time on the whole graph. We use the following {\it decomposed process}, which is a decomposition of the contact process $(X_t^{(v)})$ on $G_n$ by stochastically dominating it by a collection of processes running on local neighborhoods $\{B_u\}_u$.
 	\begin{itemize}
 		\item [1.] Initially, run $(X_{B_v}^{(v)}(t))$.
 		
 		\item[2.] In $B_v$, when some $u\in \mathcal{L}(B_v)$ becomes strongly infected in $B_v$ at time $t$ (and has been healthy until time $t-$), initiate an independent copy of $(X_{B_u}^{(u)})$ on $B_u$.
 		
 		\item[3.] Repeat Step 2 on every running copies of $(X_{B_u}^{(u)})$ until all the processes have terminated, that is, when all vertices in every generated copy are healthy.
 	\end{itemize}

 The survival time of this decomposed process stochastically dominates the process on $G_n$.
 Let ${\bf R}_v$ denote the termination time of the decomposed process and let $R_v$ be its expectation.
 Let $\textbf{U}_v$ be the enumeration of vertices $u$ of which a copy of $(X_{B_u}^{(u)})$ has been generated during the decomposed process, counted with multiplicities.
 The size of  $\textbf{U}_v$ can be controlled by ${\bf M}_v$.
 \begin{lemma} \label{lm:u}
 Let $C$ be the constant in Lemma \ref{lm:Bv}. For each $v\in V(G_n)$,
 	\begin{equation}
 		\P(|\bu_v|\ge n^{(6C+6)\delta})\le n^{-C/2}.
 	\end{equation}
 	Thus, with high probability, for all $v\in V$, we have $|\bu_v|\le n^{(6C+6)\delta}$.
 \end{lemma}
  \begin{proof}[Proof of Lemma \ref{lm:u}]
	We think about ${\bf U}_v$ as the total number of vertices in the tree with a root at $v$ and for each vertex $u$ of the tree, the number of its children equals the number of infections at the leaves of $B_u$. In particular, ${\bf U}_v$ can be thought of as a Galton-Watson tree in which ${\bf M}_v$ is the number of vertices in the first generation, denoted by $v_1, \dots, v_{k}$. Then ${\bf M}_{v_1}+\dots+{\bf M}_{v_k}$ is the number of vertices in the second generation and so on. Assume that the total number of vertices of the tree is $|\bu_v| = K$. We then have
	$$1+{\bf M}_1+{\bf M}_2 + \dots+ {\bf M}_K = K,$$
	where ${\bf M}_i$ is the number of children of the $i$-th vertex of ${\bf U}_v$. These ${\bf M}_i$ are independent and satisfies \eqref{eq:mv} by Lemma \ref{lm:Bv:Mv:1}. Note that for every $l\le K$, we have $1  +{\bf M}_1+\dots +{\bf M}_l\ge l$ as the LHS is at least the number of vertices in the tree formed by vertices $1,\dots, l$. Thus, letting $k = n^{(6C+6)\delta}$, by Lemma \ref{lm:Bv:Mv:2}, we obtain
	\begin{eqnarray}
	\P(|\bu_v|\ge n^{(6C+6)\delta})\le \P(1+{\bf M}_1+\dots+ {\bf M}_k \ge k)\le n^{-C/2}
	\end{eqnarray}
	proving the lemma.
\end{proof}

Back to the proof of Theorem \ref{thm:uper}, thanks to Lemma \ref{lm:u}, we conclude that with high probability, for all $v\in V(G_n)$, at any time $t$, the number of infected local neighborhoods $B_u$ is at most $n^{(6C+6)\delta}$. By Lemma \ref{lm:Bv:Mv}, conditioned on this event and on $(X^{(v)}_t)$, the expected survival time of $(X^{(v)}_s)_{s\ge t}$ is at most
 $$\sum _{|X^{(v)}_t\cap {\bf U}_v|} n^{\delta} \le n^{(6C+7)\delta}$$
 and so $(X_s)$ survives until time $t+ 2n^{(6C+7)\delta}$ with probability at most $1/2$. Repeating this argument for $t = a\cdot 2n^{(6C+7)\delta}$ with $a \in [1, 2, \dots, 2\log n]$, we conclude that the process $(X^{(v)}_t)$ survives until time $4 \log n \cdot n^{(6C+7)\delta}$ with probability at most $n^{-2}$.

 Taking the union bound over all $n$ vertices $v$, with probability at most $n^{-1}$, the contact process starting at all vertices infected survives until time $4 \log n \cdot n^{(7C+72)\delta}$, concluding the proof of Theorem~\ref{thm:uper}.

\section{Appendix A: Proof of Lemma \ref{lm:cyclic}: good tails}\label{app:proof}

Let $v_1\ldots v_m v_1$ be its cycle part in $\mathcal{H}^{m,l}$, and let $\mathcal{T}_{j}, \; j\in[m]$ be i.i.d. $\gw(\xi)_{\le l}$ rooted at $v_j$. Let $D_j$ be the degree of $v_j$ in $\T_{j}$, and let $\T_{j, 1},\ldots,\T_{j, D_j}$ be the subtrees rooted at $u_{j, 1},\ldots u_{j, D_j}$, the children of $v_j$ in $\mathcal{T}_j$. Then, $\dot{\cH}^{m, l}$ is simply the graph that consists of the cycle $v_1\ldots v_m v_1$ and the trees $\mathcal{T}_j$ with $j\in\{2,\ldots,m\}$.  Let $d_{j, 1}, \dots, u_{j, D_j}$ be the number of children of $u_{j, 1}, \dots, u_{j, D_j}$, respectively.

We recall the weights
	\begin{equation}
w_2(\dot{\cH}^{m, l})= \frac{  \la(D_2+2)}{1+\la(D_2+2)}, \quad w_m(\dot{\cH}^{m, l})=\frac{ \la (D_m+2)}{1+\la(D_m+2)}, \quad \wt(\dot{\cH}^{m, l})= w_2(\dot{\cH}^{m, l})+w_m(\dot{\cH}^{m, l}).\nonumber
\end{equation}

We prove Lemma \ref{lm:cyclic} by induction on $m+l$. If $m=l=1$, the result follows directly from Lemmas \ref{lm:S tail bd} and \ref{lm:M tail bd} because $\mathcal H^{m, l} \sim \gw(\xi)_{\le l}$ and $\dot{\cH}^{m, l}$ consists of a single vertex $v_1$. Assume that we have the stated result for all $m', l'$ such that $m'+l'<m+l$.


\subsection{Proof of Lemma \ref{lm:cyclic}  for $\wt(\dot{\cH}^{m, l}) S(\dot{\cH}^{m,l})$}
We first derive a recursive inequality for $S(\dot{\cH}^{m,l})$ which is similar to Lemma \ref{lm:S:recursion}.
\begin{lemma}\label{lm:rec:Sdot} We have
	\begin{eqnarray*}
		\frac{1}{2} \wt(\dot{\cH}^{m,l}) S (\dot{\cH}^{m,l}) &\le&  \frac{2\la (D_2+2)+\la^{2}(D_2+2)^{2}}{(1+\lambda (D_2+2))^{2}} +\frac{2\la (D_m+2)+\la^{2}(D_m+2)^{2}}{(1+\lambda (D_m+2))^{2}}\\
		&&\quad+  (\Pi_2-1) + (\Pi_m-1),
	\end{eqnarray*}
	where
	\begin{eqnarray}
	\Pi_{2}  &&= \left (1+ \la \wt(\dot{\cH}^{m-1, l}) S(\dot{\cH}^{m-1, l})\right )   \prod_{i=1}^{D_2} \left (1+\la w (\T_{2, i}) S(\T_{2, i})\right )\nonumber
	\end{eqnarray}
	and
	\begin{eqnarray}
	\Pi_{m}  &&= \left (1+ \la \wt(\dot{\cH}^{m-1, l}) S(\dot{\cH}^{m-1, l})\right )   \prod_{i=1}^{D_m} \left (1+\la w (\T_{m, i}) S(\T_{m, i})\right ).\nonumber
	\end{eqnarray}
\end{lemma}

\begin{proof}  Let us start by analyzing $S_2(\dot{\cH}^{m, l}) $. The analysis basically follows Lemma \ref{lm:S:recursion} except that  $v_1$ may send an infection to $v_m$ before $v_2$ sends out an infection. In that case, $S_{m} (\dot{\cH}^{m, l})$ also comes into play.
	
	Let $(X_t)$ be the root-added starred contact process on $\dot{\cH}^{m, l}$ as in Definition \ref{def:SMdef gwc2}. Let $(X_{2, t}^\sharp)$ be the process for which the healing clock at $v_2$ is disabled until all vertices of $\dot{\mathcal H}^{m, l}\setminus\{v_1, v_2\}$ are healthy. Let $S^{\sharp}_2(\dot{\cH}^{m, l})$ be the expected excursion time of $(X_{2, t}^\sharp)$. We have  $S_2(\dot{\cH}^{m, l})\le S^{\sharp}_2(\dot{\cH}^{m, l})$.

	Let $t_0$ be the first time that an event at $v_2$ happens, which must be an infection from $v_2$ to $\{v_1, v_3\}\cup \bigcup_{i=1}^{D_2} v_{2, i}$ rather than a recovery at $v_2$ because $v_2$ is initially at state $1^*$. We have
	\begin{equation}\label{eq:e:t_0}
	\E t_0 = \frac{1}{1+\lambda (D_2+2)}.
	\end{equation}

	Let $\bar t$ be the first time that $v_1$ sends a strong infection to $v_m$.  Let $t_1$ be the first time in $(X_t)$ after time $t_0$ that all vertices in $\dot{\mathcal H}^{m, l}\setminus\{v_1, v_2\}$ are healthy.
	Let $S_2^\otimes$ be the expected excursion time of the contact process $(X_{2, t}^\otimes)$ on $\dot{\mathcal H}^{m, l}$ with $v_1$ and $v_2$ always infected and one of their neighbors initially at $1^*$. For this product process, it is helpful to think about $v_1$ and $v_2$ as a giant vertex whose neighbors are $v_3, v_m$ and $v_{2, i}$ with $i=1, \dots, D_2$. Let $\theta_2^\otimes$ be the probability that an infection from $\{v_1, v_2\}$ to $\{v_3, v_m\}\cup\{v_{2, i}\}$ is a strong infection. Let $\theta_2$ be the probability that an infection from $v_2$ to $\{v_3\} \cup\{v_{2, i}\}$ is a strong infection. By writing down the formula for $\theta_2^\otimes$ and $\theta_2$ in terms of $D_m, D_3, d_{2, i}$, it is clear that
	$$\theta_2\le \frac{D_2+2}{D_2+1} \theta_2^\otimes .$$

	If $\bar t>t_0$, then
	\begin{equation}
	\E (t_1 - t_0\vert \bar t>t_0) \le \frac{D_2+1}{D_2+2} \theta_2 S_2^\otimes\le \theta_2^\otimes S_2^\otimes\nonumber,
	\end{equation}
	where $\frac{D_2+1}{D_2+2}$ is the probability that the infection from $v_2$ at time $t_0$ is not sent to $v_1$.
	
	If $\bar t<t_0$, we need to take into account the spread of infection from $v_m$ and so by using the graphical representation, we get
	\begin{equation}\label{eq:t1:t0:Sdot:2}
	\E (t_1 - t_0\vert \bar t<t_0) \le \max\{S_m(\dot{\cH}^{m, l}), \frac{D_2+1}{D_2+2}  \theta_2 S_2^\otimes \}\le S_m(\dot{\cH}^{m, l})+\theta_2^\otimes S_2^\otimes.
	\end{equation}
	Combining these equations, we get
	\begin{equation} \label{eq:t1:t0:Sdot:3}
	\E (t_1 - t_0) \le \P(\bar t<t_0)S_m(\dot{\cH}^{m, l})+ \theta_2^\otimes S_2^\otimes \le\frac{\la w_m(\dot{\cH}^{m, l})}{ 1+\la(D_2+2)}S_m(\dot{\cH}^{m, l})+ \theta_2^\otimes S_2^\otimes,
	\end{equation}
	where in the last equality, we noted that $v_1$ sends a strong infection to $v_m$ with rate  $\la w_m(\dot{\cH}^{m, l})$.

	From time $t_1$, we bound the survival time of $(X_t)_{t\ge t_1}$ by $(X_{2, t}^{\sharp})_{t\ge t_1}$ and use the same argument as for \eqref{eq:S:T} to obtain
	\begin{eqnarray*}
		&&S_2 (\dot{\cH}^{m,l})\\
		&\le&  \E t_0 + \E (t_1- t_0) + \sum _{k=0}^{\infty} \left (\frac{\la (D_2+2)}{1+\la (D_2+2)}\right )^{k} \left (\frac{1}{1+\la (D_2+2)}\right ) \left (\frac{k+1}{1+\la (D_2+2)}+ k \theta_2^\otimes S_2^\otimes  \right )\\
		&\le & \frac{1}{1+\lambda (D_2+2)} + \frac{\la w_m(\dot{\cH}^{m, l})}{1+\la(D_2+2)}S_m(\dot{\cH}^{m, l})+ \theta_2^\otimes S_2^\otimes + 1 + \la (D_2+2)  \theta_2^\otimes S_2^\otimes\\
		&= & \frac{1}{1+\lambda (D_2+2)} - \frac{1}{\la(D_2+2)} + \frac{\la w_m(\dot{\cH}^{m, l})}{1+\la(D_2+2)}S_m(\dot{\cH}^{m, l}) + \frac{1+\la(D_2+2)}{\la(D_2+2)}(1 + \la (D_2+2)  \theta_2^\otimes S_2^\otimes)
	\end{eqnarray*}

	For $S_2^{\otimes}$, we have the following analog of \eqref{eq:S:prod}
	\begin{eqnarray}
&&	1 + \la (D_2+2)\theta_2^\otimes S_2^{\otimes} = \left (1+ \frac{\la^{2} (D_3+2)}{1+\la(D_3+2)} S_3(\dot{\cH}^{m-1, l}) +\frac{\la^{2} (D_m+2)}{1+\la(D_m+2)} S_m(\dot{\cH}^{m-1, l})\right )  \nonumber\\
	&&\times \prod_{i=1}^{D_1} \left (1+\la \wt(\T_{2, i}) S(\T_{2, i})\right )= \left (1+ \la \wt(\dot{\cH}^{m-1, l}) S(\dot{\cH}^{m-1, l})\right )   \prod_{i=1}^{D_2} \left (1+\la w (\T_{2, i}) S(\T_{2, i})\right )=\Pi_{2}\label{eq:rec:S:prod:2},
	\end{eqnarray}
	where $\dot{\cH}^{m-1, l}\sim \gwc^{2}(\xi, m-1)^{l}$ contains the cycle $v_2, v_3, \dots, v_m$ and the trees $\T_{i}$ attached to $v_3, \dots, v_m$.
	
	So,
	\begin{eqnarray*}
		w_2(\dot{\cH}^{m, l}) S_2 (\dot{\cH}^{m,l}) 	&\le & \frac{-1}{(1+\lambda (D_2+2))^{2}}  + \frac{\la w_2(\dot{\cH}^{m, l})w_m(\dot{\cH}^{m, l})}{1+\la(D_2+2)} S_m(\dot{\cH}^{m, l})+  \Pi_2\\
		&\le&  \frac{-1}{(1+\lambda (D_2+2))^{2}}  + w_m(\dot{\cH}^{m, l}) S_m(\dot{\cH}^{m, l})+  \Pi_2
	\end{eqnarray*}
	Similarly, we have
	\begin{eqnarray*}
		w_m(\dot{\cH}^{m, l})	S_m (\dot{\cH}^{m,l}) 	&\le & \frac{-1}{(1+\lambda (D_m+2))^{2}}  + w_2(\dot{\cH}^{m, l})S_2(\dot{\cH}^{m, l})+  \Pi_m.
	\end{eqnarray*}
	Adding these up and arranging, we get
	\begin{eqnarray*}
		&& \wt(\dot{\cH}^{m,l}) S (\dot{\cH}^{m,l}) \le \frac{-1}{(1+\lambda (D_2+2))^{2}} + \frac{-1}{(1+\lambda (D_m+2))^{2}}+  \Pi_2 + \Pi_m \\
		&=& \frac{2\la (D_2+2)+\la^{2}(D_2+2)^{2}}{(1+\lambda (D_2+2))^{2}} +\frac{2\la (D_m+2)+\la^{2}(D_m+2)^{2}}{(1+\lambda (D_m+2))^{2}}+  (\Pi_2-1) + (\Pi_m-1).
	\end{eqnarray*}
completing the proof of the recursion.
\end{proof}

	By induction hypothesis and Proposition \ref{prop:tail}, $\Pi_2-1$ and $\Pi_m-1$ have strong tails. So does the constant part. Thus, $\wt(\dot{\cH}^{m,l})	S (\dot{\cH}^{m,l})$ has a good tail, proving Lemma \ref{lm:cyclic}  for $\wt(\dot{\cH}^{m, l}) S(\dot{\cH}^{m,l})$.

\subsection{Proof of Lemma \ref{lm:cyclic} for $\wt(\cH^{m, l})S({\mathcal{H}}^{m,l})$ }\label{subsec:prof:uni S}

By the same argument as in Lemmas \ref{lm:S:recursion} and \ref{lm:rec:Sdot}, we obtain the following recursion.
\begin{lemma}\label{lem:gwc1 recursion atypical}
	Under the hypothesis of Lemma \ref{lm:cyclic}, it holds that
	\begin{eqnarray}\label{eq:gwc1 Srecursion atypical}
	S(\mathcal{H}^{m,l}) \le \frac{1}{1+\la (D_1+3)} -\frac{1}{\la(D_1+3)}+  \wt(\mathcal{H}^{m,l})^{-1}\left (1+ \la \wt(\dot{\cH}) S(\dot{\cH})\right ) \prod_{i=1}^{D_1} \left (1+\la \wt(\T_{1, i}) S(\T_{1, i})\right )\nonumber.
	\end{eqnarray}
\end{lemma}

By the induction hypothesis, each $\wt(\T_{1, i}) S(\T_{1, i})$ has a good tail. As proved in the previous section, $\wt(\dot{\cH}) S(\dot{\cH})$ also has a good tail. Combining this with Lemma \ref{lem:gwc1 recursion atypical} and Proposition \ref{prop:tail} as in the proof of Lemma \ref{lm:S tail bd}, we conclude that $\wt(\cH^{m, l})S(\mathcal{H}^{m,l})$ has a good tail.

\subsection{Proof of Lemma \ref{lm:cyclic} for $(1.5)^{l-1} \wt(\dot{\cH}^{m, l}) M_{ l}(\dot{\cH})$}\label{subsubsec:unic3}

Following the proof of Lemma \ref{lm:Mrecursion}, we derive a recusive bound for $M_{ l}(\dot{\cH})$.
\begin{lemma}\label{lm:recur:dot M:Hml} 	Under the hypothesis of Lemma \ref{lm:cyclic}, it holds that
	\begin{eqnarray*}
		\frac{1}{2}	\wt(\dot{\cH}^{m, l})M_{l}(\dot{\cH}^{m, l}) &\le& \la(\Sigma_2 +  \Sigma_m),
	\end{eqnarray*}
	where
	\begin{eqnarray*}
		\Sigma_2 &=&\left [1+\la \wt(\dot{\cH}^{m-1, l}) S(\dot{\cH}^{m-1, l}) \right ]\Bigg[ \sum_{i=1}^{D_2}   w_{2, i} M_{l-1}(\T_{2, i})\prod_{j\neq i, 1\le j\le D_2} (1+\la \wt(\T_{2, j}) S(\T_{2, j})) \Bigg] \\
		&&  +  \wt(\dot{\cH}^{m-1, l}) M_{l}(\dot{\cH}^{m-1, l})   \prod_{j=1}^{D_2} (1+\la \wt(\T_{2, j}) S(\T_{2, j})),
	\end{eqnarray*}
	and similarly for $\Sigma_m$.
\end{lemma}

\begin{proof} [Proof of Lemma \ref{lm:recur:M:Hml}]
	We shall reuse the notations in the proof of Lemma \ref{lm:rec:Sdot}. We denote the corresponding expected total number of infections at depth-$l$ leaves of the processes $(X_{2, t}^{\sharp})$ and $(X_{2, t}^\otimes)$ by $ M_2^\sharp$ and $M_{2}^{\otimes}$.

	Similar to the derivation of \eqref{eq:Msharp and Motimes}, we observe that after the recovery of the first infection from $v_2$ in $(X^\sharp_{2, t})$, the number of excursions of $(X_{2, t}^\otimes)$ included in a single excursion of $(X^\sharp_{2, t})$ is the same as a geometric random variable with success probability $(1+\lambda (D_2+2) \theta_2^\otimes)^{-1}$. Therefore, the expected number of infections at depth-$l$ leaves after time $t_1$ is
	\begin{equation}
	\lambda (D_2+2)  \theta_2^\otimes M_{2}^\otimes(\dot{\cH}^{m,l}) \label{eq:Mdot:1}.
	\end{equation}
	By the same argument as in \eqref{eq:t1:t0:Sdot:3}, we bound from above the expected number of infections at depth-$l$ leaves of $(X_t)$ during $[0, t_1]$ by
	\begin{equation}
	\frac{\la \wt_m(\dot{\cH}^{m, l})}{1+\la(D_2+2)}M_{l, m}(\dot{\cH}^{m, l})+ \theta_2^\otimes M_2^\otimes  \label{eq:Mdot:2}.
	\end{equation}

	Now we attempt to control $M^\otimes(\dot{\cH}^{m,l})$ in terms of $\{M(\T_{2, i}) \}$ and $M(\dot{\cH}^{m-1, l})$ where $\dot{\cH}^{m-1, l}\sim \gwc^{2}(\xi, m-1)^{l}$ contains the cycle $v_2, v_3, \dots, v_m$ and the trees $\T_{i}$ attached to $v_3, \dots, v_m$. Let $\mathcal L$ be the collection of depth-$l$ leaves of $\dot{\cH}^{m,l}$.
	We observe that
	\begin{equation*}
	\begin{split}
	\lim_{t_0\rightarrow \infty} \frac{1}{t_0} \sum_{v\in\mathcal{L}}\E_{\textsc{cp}} &
	\left[\,
	\left|\left\{ s \in [0,t_0]:\; X_{2, s}^\otimes(v)=1^* \textnormal{ and } X_{2, s-}^\otimes (v)=0 \right\} \right| \,
	\right]\\
	&=
	\frac{M_2^\otimes(\dot{\cH}^{m,l})}{(\lambda (D_2+2)\theta_2^\otimes)^{-1} + S_2^\otimes(\dot{\cH}^{m,l})}.
	\end{split}
	\end{equation*}
	On the other hand, by splitting $\mathcal L$ into parts corresponding to $\T_{2, i}$ and $\dot{\cH}^{m-1, l}$, the above is also equal to
	\begin{equation*}
	\sum_{i=1}^{D_2} \frac{M_{l-1}(\T_{2, i})}{(\lambda\wt(\T_{2, i}))^{-1} + S(\T_{2, i})}+\frac{M_{l}(\dot{\cH}^{m-1, l})}{(\la \wt(\dot{\cH}^{m-1, l}))^{-1} + S(\dot{\cH}^{m-1, l})}.
	\end{equation*}
	
	Combining these equations and \eqref{eq:rec:S:prod:2}, we get
	\begin{eqnarray*}
		&&(D_2+2)\theta_2^\otimes M_2^\otimes(\dot{\cH}^{m,l})\\
		& =&   \left [1+\la \wt(\dot{\cH}^{m-1, l}) S(\dot{\cH}^{m-1, l}) \right ]\Bigg[ \sum_{i=1}^{D_2}  \wt(\T_{2, i}) M_{l-1}(\T_{2, i})\prod_{j\neq i, 1\le j\le D_2} (1+\la \wt(\T_{2, j}) S(\T_{2, j})) \Bigg] \\
		&&  +  \wt(\dot{\cH}^{m-1, l}) M_{l}(\dot{\cH}^{m-1, l})   \prod_{j=1}^{D_2} (1+\la \wt(\T_{2, j}) S(\T_{2, j})) = \Sigma_2.\nonumber
	\end{eqnarray*}
	By this together with \eqref{eq:Mdot:1}, \eqref{eq:Mdot:2}, and $M^\sharp(\dot{\cH}^{m,l}) \geq M_{ l}(\dot{\cH}^{m,l})$, we conclude  that
	\begin{eqnarray*}
		M_{l, 2}(\dot{\cH}^{m, l}) &\le&  \frac{\la \wt_m(\dot{\cH}^{m, l})}{1+\la(D_2+2)}M_{l, m}(\dot{\cH}^{m, l}) + \frac{1+\la(D_2+2)}{D_2+2} \Sigma_2.
	\end{eqnarray*}
	Thus,
	\begin{eqnarray*}
		\wt_2(\dot{\cH}^{m, l}) M_{l, 2}(\dot{\cH}^{m, l}) &\le& \frac{1}{2} \wt_m(\dot{\cH}^{m, l})M_{l, m}(\dot{\cH}^{m, l}) +\la\Sigma_2.
	\end{eqnarray*}
	Adding this with the analog for $M_{l, m}(\dot{\cH}^{m, l})$ and rearranging, we obtain
	\begin{eqnarray*}
	 \wt(\dot{\cH}^{m, l})M_{l}(\dot{\cH}^{m, l}) &\le& 2\la\Sigma_2 +2\la \Sigma_m
	\end{eqnarray*}
	as claimed.
\end{proof}

To conclude that $(1.5)^{l-1} \wt(\mathcal H^{m, l}) M_l(\mathcal H^{m, l})$ has a good tail, we proceed as  in the proof of Lemma \ref{lm:M tail bd}. We set
\begin{equation*}
W_i = \max\{(1.5)^{l-1} \wt(\T_{2, i})M_{l-1}(\T_{2, i}),  \wt(\T_{2, i}) S(\T_{2, i})\}, \quad i=1, \dots, D_2
\end{equation*}
and
\begin{equation*}
W_{0} = \max\left \{(1.5)^{l-1}\wt(\dot{\cH}^{m-1, l}) M_{l}(\dot{\cH}^{m-1, l}),  \wt(\dot{\cH}^{m-1, l}) S(\dot{\cH}^{m-1, l})\right  \},
\end{equation*}
where we used the subscript $0$ in place of $\dot{\mathcal H}^{m-1, l}$ for notational convenience.
Then, the quantity $\Sigma_2$ in Lemma \ref{lm:recur:dot M:Hml} satisfies
\begin{equation*}
2(1.5)^{l-1}\la \Sigma_2 \le 3\sum_{i=0}^{D_2} \la W_i \prod_{0\le j\le D_2, j\neq i} (1+\la W_{j})\le \frac{3}{2}\left (\prod_{i=0}^{D_2} (1+2\la W_{i})-1\right ).
\end{equation*}

We have by the induction hypothesis that $W_i$ has a good tail for all $0\le i\le D_2$. By  Proposition \ref{prop:tail},  $\la\Sigma_2$ has a strong tail. Similarly, $\la\Sigma_m$ also has a strong tail. And so does, $(1.5)^{l-1} \wt(\dot{\cH}^{m, l}) M_l(\dot{\cH}^{m, l})$.

\subsection{Proof of Lemma \ref{lm:cyclic} for $(1.5)^{l}\wt(\mathcal H^{m, l})M_{ l}(\mathcal{H}^{m,l})$ }\label{subsec:prof:uni M}

For the tail of $M_{ l}(\mathcal{H}^{m,l})$, we have the following recursive inequality which can be derived analogously as Lemma \ref{lm:Mrecursion}.
\begin{lemma}\label{lm:recur:M:Hml} 	Under the hypothesis of Lemma \ref{lm:cyclic}, we have
\begin{eqnarray}
&& \wt(\cH^{m, l})  M_{ l}(\mathcal{H}^{m,l}) \le \wt(\dot{\cH}^{m, l}) M_{l}(\dot{\cH}^{m, l}) \prod_{j=1}^{D_1} \left (1+\la \wt(\T_{1, j}) S(\T_{1, j})\right ) \nonumber\\
  &&\quad	\quad\quad+\left [1+ \la \wt(\dot{\cH}^{m, l}) S(\dot{\cH}^{m, l})    \right ]  \left [ \sum_{i=1}^{D_1}   \wt(\T_{1, i}) M_{l-1}(\T_{1, i})\prod_{j\neq i, 1\le j\le D_1} \left (1+\la \wt(\T_{1, j}) S(\T_{1, j})\right ) \right ].\nonumber
\end{eqnarray}

\end{lemma}
From here, the conclusion that $(1.5)^{l}\wt(\mathcal H^{m, l}) M_{l}(\mathcal H^{m, l})$ has a good tail follows from the same argument as above. This completes  the proof of Lemma \ref{lm:cyclic}.

\section{Appendix B: Proof of Lemma \ref{lm:tree:bound}}\label{app:treebound}
	Let $N_i\sim \widetilde \mu$ be independent. We have for all $s\ge 1$,
	\begin{eqnarray}
	\E \exp(tZ_{s+1} \tilde d^{-(s+1)}) &=& \E  \exp\left (\sum_{i=1}^{Z_s} t N_i \tilde d^{-(s+1)}\right ) =\E \left [\E  \exp\left (\sum_{i=1}^{Z_s} t N_i \tilde d^{-(s+1)}\right )\big\vert Z_s\right  ]\nonumber\\
	&=&  \E \left [\left (\E  \exp\left (t N_1 \tilde d^{-(s+1)}\right )\right )^{Z_s} \right  ]=  \E \exp\left (Z_s \log \left (\E  \exp\left (t N_1 \tilde d^{-(s+1)}\right )\right )  \right  ).\nonumber
	\end{eqnarray}
	Since $\E_{N\sim \mu} e^{cN}<\infty$, we have $\E_{N\sim \widetilde \mu} e^{uN}<\infty$ for all $u<c$. Let $b=c/2$.
	For all $u\in (0, b]$, we have
	\begin{eqnarray}
	\E e^{uN_1} &=& \E \sum_{i=0}^{\infty} \frac{u^{i}N^{i}}{i!} \le 1 + u\widetilde d + \frac{u^{2}}{b^{2}}\sum_{i=2}^{\infty}\frac{b^{i} N^{i}}{i!} \nonumber \\
	&\le& e^{u\widetilde d} + \frac{u^{2}}{b^{2}} \E e^{bN_1} \le e^{u\widetilde d} +  \exp\left (\frac{u^{2}}{b^{2}} \E e^{bN_1}\right )-1\le \exp\left (u\widetilde d(1+u\alpha)\right )\label{eq:uN},
	\end{eqnarray}
	where $\alpha$ is a large constant. Let $t>0$ be such that
	$$\exp\left (2 \alpha t\sum_{i=1}^{\infty} \widetilde d^{-i}\right )<2\quad \text{and}\quad 2t\le b.$$
	We define $t_n = t\exp\left (2 \alpha t\sum_{i=n+1}^{\infty} \widetilde d^{-i}\right )\le t\exp\left (2 \alpha t\sum_{i=1}^{\infty} \widetilde d^{-i}\right )<2t\le b$. Then, for all $s\ge 1$,
	\begin{eqnarray*}
		\E \exp(t_{s+1}Z_{s+1} \tilde d^{-(s+1)}) &=& \E \exp\left (Z_s \log \left (\E  \exp\left (t_{s+1} \tilde d^{-(s+1)}N_1 \right )\right )  \right  )\\
		&\le&  \E \exp\left (Z_s \widetilde d^{-s}t_{s+1}  (1+\alpha t_{s+1} \tilde d^{-(s+1)}) \right  )\quad \text{ by \eqref{eq:uN}}\\
		&\le&  \E \exp\left (Z_s \widetilde d^{-s}t_{s+1} (1+2 t \alpha  \tilde d^{-(s+1)}) \right  )\quad \text{ because $t_{s+1}\le 2t$}\\
		&\le&  \E \exp\left (t_{s}Z_s \widetilde d^{-s} \right  ).
	\end{eqnarray*}
	Thus,
	$$\sup_{s\ge 2} \E \exp\left (t_{s}Z_s \widetilde d^{-s} \right  ) = \E \exp\left (t_{1}Z_1 \widetilde d^{-1} \right  ) <\infty,$$
	for $t$ sufficiently small. This completes the proof.

\section{Acknowledgments} We thank the anonymous referees for their thorough reviews and suggestions.

\bibliographystyle{plain}
\bibliography{reference}
\end{document}